\documentclass[12pt]{amsart}
\usepackage{a4wide}
\usepackage[T1]{fontenc}
\usepackage{amssymb,amsmath,amsthm,latexsym}
\usepackage{mathrsfs}
\usepackage[usenames,dvipsnames]{color}
\usepackage{euscript}
\usepackage{graphicx}
\usepackage{mdwlist}
\usepackage{enumerate}
\usepackage{enumitem}
\usepackage{xcolor}
\usepackage{mathtools,dsfont,wasysym}
\usepackage{stmaryrd}
\usepackage{ulem}
\usepackage{hyperref}
\usepackage{cancel}
\usepackage{float}

\hypersetup{colorlinks=true, linkcolor=blue, citecolor=red, urlcolor=cyan}

\newtheorem{theorem}{Theorem}[section]
\newtheorem{lemma}[theorem]{Lemma}
\newtheorem{corollary}[theorem]{Corollary}
\newtheorem{cor}[theorem]{Corollary}
\newtheorem{proposition}[theorem]{Proposition}

\newcounter{maintheorem}

\theoremstyle{remark}
\newtheorem{remark}[theorem]{Remark}
\newtheorem{question}[theorem]{Question}

\theoremstyle{definition}
\newtheorem{definition}[theorem]{Definition}

\newtheorem{example}[theorem]{Example}

\numberwithin{equation}{section}
\makeatother

\makeatletter
\renewcommand{\tocsection}[3]{%
\indentlabel{\@ifnotempty{#2}{\bfseries\ignorespaces#1 #2\quad}}\bfseries#3}
\renewcommand{\tocsubsection}[3]{%
\indentlabel{\@ifnotempty{#2}{\ignorespaces#1 #2\quad}}#3}

\newcommand\@dotsep{4.5}
\def\@tocline#1#2#3#4#5#6#7{\relax
\ifnum #1>\c@tocdepth 
\else
\par \addpenalty\@secpenalty\addvspace{#2}%
\begingroup \hyphenpenalty\@M
\@ifempty{#4}{%
\@tempdima\csname r@tocindent\number#1\endcsname\relax
}{%
\@tempdima#4\relax
}%
\parindent\z@ \leftskip#3\relax \advance\leftskip\@tempdima\relax
\rightskip\@pnumwidth plus1em \parfillskip-\@pnumwidth
#5\leavevmode\hskip-\@tempdima{#6}\nobreak
\leaders\hbox{$\m@th\mkern \@dotsep mu\hbox{.}\mkern \@dotsep mu$}\hfill
\nobreak
\hbox to\@pnumwidth{\@tocpagenum{\ifnum#1=1\bfseries\fi#7}}\par
\nobreak
\endgroup
\fi}
\AtBeginDocument{%
\expandafter\renewcommand\csname r@tocindent0\endcsname{0pt}
}
\def\l@subsection{\@tocline{2}{0pt}{2.5pc}{5pc}{}}
\makeatother

\usepackage{todonotes}

\newcommand{\sna}{\operatorname{SNA}}
\newcommand{\SNA}{\operatorname{SNA}}
\newcommand{\bbr}{\mathbb{R}}
\newcommand{\bbn}{\mathbb{N}}

\newcommand{\N}{\mathbb{N}}

\newcommand{\calf}{\mathcal{F}}
\newcommand{\call}{\mathcal{L}}
\newcommand{\fm}{\calf(M)}
\newcommand{\na}{\operatorname{NA}}
\newcommand{\dens}{\operatorname{dens}}

\newcommand{\pna}{\operatorname{PNA}}
\newcommand{\PNA}{\operatorname{PNA}}

\newcommand{\ldira}{\operatorname{LDirA}}

\newcommand{\dira}{\operatorname{DirA}}

\newcommand{\lipa}{\operatorname{LipA}}

\newcommand{\lip}{\operatorname{Lip}_0}
\newcommand{\Lip}{\operatorname{Lip}_0}
\newcommand{\naf}{\na(\fm)}
\newcommand{\der}{\operatorname{Der}}
\newcommand{\D}{\operatorname{Der}}
\newcommand{\sign}{\operatorname{sign}}
\newcommand{\rr}{(\bbr)}

\newcommand{\F}{\mathcal{F}}
\newcommand{\eps}{\varepsilon}

\DeclareMathOperator{\conv}{conv}

\renewcommand{\phi}{\varphi}
\renewcommand{\epsilon}{\varepsilon}

\renewcommand{\subset}{\subseteq}

\newlength\Colsep
\begin{document}
\title[Embeddings in the sets of norm-attaining Lipschitz functions]{Embeddings of infinite-dimensional spaces in the sets of norm-attaining Lipschitz functions}

\author[Choi]{Geunsu Choi}
\address[Choi]{Department of Mathematics Education, Sunchon National University, 57922 Jeonnam, Republic of Korea \newline
\href{http://orcid.org/0000-0002-4321-1524}{ORCID: \texttt{0000-0002-4321-1524}}}
\email{\texttt{gschoi@scnu.ac.kr}}

\author[Jung]{Mingu Jung}
\address[Jung]{School of Mathematics, Korea Institute for Advanced Study, 02455 Seoul, Republic of Korea\newline
\href{https://orcid.org/0000-0003-2240-2855}{ORCID: \texttt{0000-0003-2240-2855}}}
\email{jmingoo@kias.re.kr}

\author[Lee]{Han Ju Lee}
\address[Lee]{Department of Mathematics Education, Dongguk University, 04620 Seoul, Republic of Korea \newline
\href{https://orcid.org/0000-0001-9523-2987}{ORCID: \texttt{0000-0001-9523-2987}}}
\email{\texttt{hanjulee@dgu.ac.kr}}

\author[Rold\'an]{\'Oscar Rold\'an}
\address[Rold\'an]{Department of Mathematics Education, Dongguk University, 04620 Seoul, Republic of Korea \newline
\href{https://orcid.org/0000-0002-1966-1330}{ORCID: \texttt{0000-0002-1966-1330}}}
\email{\texttt{oscar.roldan@uv.es}}


\date{\today}                                           


\begin{abstract}
Motivated by the result \cite{DMQR23} that there exist metric spaces for which the set of strongly norm-attaining Lipschitz functions does not contain an isometric copy of $c_0$, we introduce and study a weaker notion of norm-attainment for Lipschitz functions called the pointwise norm-attainment. As a main result, we show that for every infinite metric space $M$, there exists a metric space $M_0 \subseteq M$ such that the set of pointwise norm-attaining Lipschitz functions on $M_0$ contains an isometric copy of $c_0$. 
We also observe that there are countable metric spaces $M$ for which the set of pointwise norm-attaining Lipschitz functions contains an isometric copy of $\ell_\infty$, which is a result that does not hold for the set $\sna(M)$ of strongly norm-attaining Lipschitz functions. Several new results on $c_0$-embedding and $\ell_1$-embedding into the set $\sna(M)$ are presented as well. In particular, we show that if $M$ is a subset of an $\bbr$-tree containing all the branching points, then $\sna(M)$ contains $c_0$ isometrically. As a related result, we provide an example of metric space $M$ for which the set of norm-attaining functionals on the Lipschitz-free space over $M$ cannot contain an isometric copy of $c_0$. 
Finally, we compare the concept of pointwise norm-attainment with the several different kinds of norm-attainment from the literature.
\end{abstract}

\maketitle

\hypersetup{linkcolor=black}

\makeatletter \def\l@subsection{\@tocline{2}{0pt}{1pc}{5pc}{}} \def\l@subsection{\@tocline{2}{0pt}{3pc}{6pc}{}} \makeatother

\tableofcontents

\hypersetup{linkcolor=blue}

\section{Introduction}

In this article, a new norm-attainment notion for Lipschitz functions, namely the \textit{pointwise norm-attainment}, will be introduced and studied. On top of that, we will discuss spaceability questions for sets of real Lipschitz functions that attain their norms in several different ways.

We use standard notations and terminology for Banach spaces (see, for instance, \cite{FHHMZ11}). In particular, for a real Banach space $X$, let $X^*$, $B_X$, and $S_X$ respectively denote its topological dual, its closed unit ball, and its unit sphere. Given two Banach spaces $X$ and $Y$, let $\call(X, Y)$ denote the space of bounded and linear operators from $X$ to $Y$. Given an operator $T\in\call(X, Y)$, we say that it is \textit{norm-attaining} if there exists some point $x\in S_X$ such that $\|T(x)\|=\|T\|$. The set of norm-attaining operators from $X$ to $Y$ is denoted by $\na(X, Y)$. When the range space is $Y = \bbr$, we write them by $\call(X)$ and $\na (X)$ for simplicity. 

In recent years, the study of the lineability of sets has gained a lot of attention  \cite{ABPS, AGS, BPS, GQ}. We say that a subset $A$ of a vector space is \textit{$n$-lineable} ($n\in\bbn$) if $A\cup\{0\}$ contains some $n$-dimensional linear space, it is \textit{lineable} if $A\cup\{0\}$ contains some infinite-dimensional linear space, and it is \textit{spaceable} if $A\cup\{0\}$ contains some closed infinite-dimensional linear space. It is clear that, for instance, if $X = Z^*$ is an infinite-dimensional dual Banach space, then $\na(X)$ is spaceable as $Z$ is contained in $\na(X)$, if $X=c_0$, then $\na(X)$ is a non-closed infinite-dimensional linear space, and if $X=\ell_1$, then $\na(X)$ is not a linear space (see for instance \cite{JMR23}). Despite that, for $X=\ell_1$, like for many other spaces, $\na(X)$ is at least $2$-lineable. In 2001, G. Godefroy asked if for every Banach space $X$ of dimension at least $2$, the set $\na(X)$ is always $2$-lineable (see \cite[Problem III]{Godefroy01}). This was solved in the negative by R. Rmoutil in 2017 (see \cite{Rmoutil17}): there is a renorming $X$ of $c_0$ due to C. Read (see \cite{Read18}) for which $\na(X)$ is not 2-lineable. In this paper, we study related spaceability questions for several sets of real Lipschitz functions. 

Let $M$ be a pointed metric space (that is, a metric space $M$ with a distinguished point $0$), and let $Y$ be a real Banach space. We denote by $\lip(M, Y)$ the Banach space of Lipschitz functions $f:M\rightarrow Y$ such that $f(0)=0$ endowed with the Lipschitz norm
$$\|f\|:=\sup\left\{ \frac{\|f(q)-f(p)\|}{d(p, q)}:\, (p,q) \in \widetilde{M} \right\},$$
where the notation $\widetilde{A}$ means the set $\{(p,q)\in A \times A : p\neq q \}$ for a given set $A$. When the range space is $Y = \bbr$, we omit to include $\bbr$ in the notation for all kinds of sets of specific Lipschitz functions from $M$ to $\bbr$. For convenience, let us write the slope of $f$ from a point $p$ to a point $q$ in $M$ as
\[
S(f,p,q) = \frac{f(q)-f(p)}{d(p,q)}. 
\]
We say that a Lipschitz function $f\in\lip(M, Y)$ \textit{strongly attains its norm} if there exist $(p,q) \in \widetilde{M}$ such that $\|S(f,p,q)\|=\|f\|$. The set of strongly norm-attaining Lipschitz functions from $M$ to $Y$ is denoted by $\sna(M,Y)$, and it has been studied extensively in recent years (see for instance \cite{CCGMR19, Chiclana22, CGMR21, CM19, GPPR18, Godefroy16, JMR23, KMS16} and the references therein). In this article, unless specified otherwise, every metric space will be assumed to be pointed, and note that for our purposes the choice of the point $0$ is irrelevant since if $0$ and $0'$ are two distinguished points for $M$, the correspondence between $\lip(M, Y)$ and $\operatorname{Lip}_{0'}(M, Y)$ given by $f\mapsto f-f(0')$ is an isometric isomorphism that preserves the norm and the norm-attainment behaviour of every Lipschitz function in the space.


For simplicity, from this point onwards, we will use the following notational abuse: we will say that a Banach space $X$ is isomorphically (respectively, isometrically) contained (or embedded) in a subset $A$ of a Banach space $Y$ if the set $A$ contains a space which is linearly isomorphic (respectively, isometrically isomorphic) to $X$. 


In \cite{CDW16} and \cite{CJ17} it was respectively shown that $\ell_\infty$ can always be embedded isomorphically and isometrically inside the space $\lip(M)$ for any infinite metric space $M$. Soon after, spaceability questions were studied in detail for the set $\sna(M)$ in \cite{AMRT23,DMQR23,KR22}. We summarize the main results from those papers regarding the existence and sizes of spaces within $\sna(M)$. Recall that a metric space $M$ is \textit{discrete} if it contains no accumulation points, and it is \textit{uniformly discrete} if $\inf\{d(p, q):\, (p,q) \in \widetilde{M}\}>0$. 

\begin{enumerate}[label=(S\arabic*)]
\itemsep0.3em 
\item If $M$ has finite cardinality $|M|=n\in\bbn$, then $\sna(M)=\lip(M)$ is an $(n-1)$-dimensional Banach space, so we are only interested in infinite metric spaces.
\item If $M$ is infinite, then $\ell_1^n$ is isometrically contained in $\sna(M)$ for all $n\in\bbn$ (\cite[Theorem 1]{KR22}).
\item If $Y$ is any Banach space, then $Y$ is isometrically contained in $\sna(B_{Y^*})$ (\cite[Proposition 1]{KR22}). \label{S-Nec-Ell1}
\item If $Y$ is isometrically contained in $\sna(M)$, then $Y^*$ is separable if and only if $M$ is separable (\cite[Theorem 2]{KR22}). \label{KRS} 
\item If $M$ is $\sigma$-precompact and $Y$ is isometrically contained in $\sna(M)$, then $Y$ is isomorphically polyhedral and has separable dual (\cite[Theorem 3]{KR22}).
\item If $M$ is infinite, then $c_0$ is isomorphically contained in $\sna(M)$ (this is proven for complete spaces in \cite[Main Theorem]{AMRT23}, but in the non-complete case it is also true, see \cite[Theorem 4.2]{DMQR23}). \label{S} 
\item If $M$ is infinite and not uniformly discrete, then $c_0$ is isometrically contained in $\sna(M)$ (\cite[Theorem 4.2]{DMQR23}). \label{SS}
\item There exist uniformly discrete metric spaces $M$ such that $c_0$ cannot be isometrically contained in $\sna(M)$ (\cite[Theorems 4.1 and 4.4]{DMQR23}). \label{SSS}
\item If the set of accumulation points of $M$ has density character $\Gamma$, for some infinite cardinal $\Gamma$, then $\sna(M)$ contains $c_0(\Gamma)$ isometrically (\cite[Theorem 5.2]{DMQR23}). \label{SSSS}
\end{enumerate}

In particular, if $M$ is an infinite metric space, the set $\sna(M)$ is spaceable, as it contains an isomorphic copy of $c_0$, in contrast to what happens in the case of norm-attaining functionals. Besides the strong norm-attainment, several other natural norm-attainment notions have been extensively studied for Lipschitz functions, namely $\der(X)$, $\ldira(X)$, $\dira(X)$, and $\lipa(M)$ (see Section \ref{Section-NA-Lips} for the definitions and \cite{Choi23, CCM20, Godefroy16, KMS16} for the background). Note that by the natural set inclusions from \cite[Section 1]{CCM20}, the sets $\der(X)$, $\ldira(X)$, $\dira(X)$, and $\lipa(M)$ always contain $c_0$ isometrically by \ref{SS} and \cite[Theorem 5]{CJ17}. In this paper, we introduce a new natural norm-attainment notion that is weaker than the strong norm-attainment, namely the pointwise norm-attainment of Lipschitz functions (see Definition \ref{def:pna}). We will see in Theorem \ref{theorem:negative-example-ell1} that the set of pointwise norm-attaining Lipschitz functions on $M$, which we denote by $\pna(M)$, needs not contain $\ell_\infty$ isometrically, which further motivates our main question: whether or not it is always possible to isometrically embed $c_0$ in $\pna(M)$ (see Question \ref{Q:Main}). Our main result, Theorem \ref{Main-Theorem-c0-PNA} provides a partial positive answer to this question. We also observe that for a metric space $M$ such that $c_0$ is known not to be isometrically embedded into the set $\sna(M)$, the space $c_0$ can in fact be isometrically embedded if one considers the set $\pna (M)$ instead. 


The paper is structured as follows. First, we introduce the main concept of the article, namely the \textit{pointwise norm-attainment} of Lipschitz functions (see Definition \ref{def:pna}), and study spaceability questions for the set $\pna(M)$ of pointwise norm-attaining Lipschitz functions, and this study is split into several sections. In Section \ref{Section-PNA-Ell-infty}, we provide preliminary spaceability results for the set $\pna(M)$. We will show in particular that sometimes it is possible to embed $\ell_\infty$ in $\pna(M)$ but it is not always the case. We will also show that removing one single point from a space $M$ can completely change the spaceability properties of $\pna(M)$. 
In Section \ref{Section-SNA-c0}, we present several new positive results on the isometric embedding of $c_0$ into $\sna(M)$. Moreover, we relate the isometric embedding of $c_0$ into $\sna(M)$ with the isometric embedding of $\ell_1$ into $\fm$, which allows us to present new examples of metric spaces $M$ for which $\sna(M)$ does not contain $c_0$ isometrically. We also show that if $M$ is a subset of an $\bbr$-tree containing all the branching points, then $c_0$ is isometrically contained in $\sna(M)$. We finish the section by showing that there exists a metric space $M$ such that $c_0$ is not contained in $\na(\fm)$ with the aid of some recent results about convex integrals of molecules. 
In Section \ref{Section-PNA-c0}, we tackle our main question: whether or not it is possible to isometrically embed $c_0$ in $\pna(M)$ for any infinite metric space $M$ (see Question \ref{Q:Main}). Throughout the section, we will provide several constructions of an isometric copy of $c_0$ in $\pna(M)$ for several classes of metric spaces $M$, which contrast with some of the negative results for $\sna(M)$ from Section \ref{Section-SNA-c0}. Our main result, Theorem \ref{Main-Theorem-c0-PNA}, states that if $M$ is infinite, then it contains a metric subspace $M_0\subset M$ such that $c_0$ is isometrically contained in $\pna(M_0)$. In Section \ref{Section-PNA-Ell-1} we will discuss some results about the possibility of embedding $\ell_1$ in $\sna(M)$ and in $\pna(M)$, remarking the differences between both sets once more. Finally, in Section \ref{Section-NA-Lips} as an appendix, we will recall several norm-attainment notions for Lipschitz functions from the literature, and study spaceability questions and the set relations among them.

\section{Pointwise norm-attainment and preliminary results}\label{Section-PNA-Ell-infty}

We will first introduce the main concept of the article.

\begin{definition}\label{def:pna}
Let $M$ be a pointed metric space, and let $Y$ be a real Banach space. We say that $f \in \Lip(M,Y)$ \textit{attains its pointwise norm} if there exists $p \in M$ such that
$$\sup_{q \in M \setminus \{p\}} \| S(f,p,q)\| = \|f\|.$$
In this case, we write $f \in \PNA(M,Y)$.
\end{definition}

By definition, the inclusions
$$\SNA(M,Y) \subseteq \PNA(M,Y) \subseteq \Lip(M,Y)$$
clearly hold. The following example, which will appear again later in Theorem \ref{theorem:negative-example-ell1}, may help to understand better that the pointwise norm attainment is a natural notion, which lies in between the set of strongly norm attaining Lipschitz functions and the space of Lipschitz functions.

\begin{example}\label{example:discrete-metric}
Let $M$ be the usual discrete infinite metric space, that is, $d(p,q) = 1$ for all $(p,q) \in \widetilde{M}$, and let $f \in \lip(M)$ be given. Then, $f \in \sna(M)$ if and only if both $\sup \{ f(p) : {p \in M}\}$ and $\inf \{ f(p) : {p \in M}\}$ are attained. In the same manner, $f \in \pna(M)$ if and only if either one of $\sup \{ f(p) : {p \in M}\}$ or $\inf \{ f(p) : {p \in M}\}$ is attained.
\end{example}

Note that this example shows, in particular, that the previous inclusions cannot be reversed in general.
Indeed, as we will show from several spaceability results in the following sections, the concepts of strong norm-attainment and pointwise norm-attainment are distinct. 
In the Appendix (Section \ref{Section-NA-Lips}) we study the relations between the notion of pointwise norm-attainment and other norm-attainment notions, which may provide a better intuition for the geometric relations between the new concept and the known ones. Moreover, we will see later that in many cases there is a Lipschitz function which does not attain its pointwise norm. Therefore, the following question arises naturally.
\begin{question}\label{Q:Non-triviality}
Is the inclusion $\pna(M)\subset \lip(M)$ always strict for every infinite metric space $M$?
\end{question}
Recall from \ref{KRS} that if $M$ is separable, any Banach space in $\sna(M)$ must have separable dual. Since the set $\pna(M)$ inherits all the positive spaceability results from $\sna(M)$, it is natural to wonder if a version of \ref{KRS} is true for $\pna(M)$. This, combined with the fact that often $\pna\subsetneq \lip(M)$, motivates the following questions.
\begin{question}\label{Q:PNA-Ell-Infty}
Is $\ell_\infty$ isometrically contained in $\pna(M)$ for every infinite metric space $M$?
\end{question}
\begin{question}\label{Q:PNA-Size}
If $M$ is a separable infinite metric space and $Y$ is a Banach space contained in $\pna(M)$, must $Y$ have a separable dual?
\end{question}

We will see in this section that the answers to all the questions above are negative. 
We first present the following result which answers Questions \ref{Q:Non-triviality} and \ref{Q:PNA-Size} in the negative.




\begin{proposition}\label{Countable-M-EllInfty}
There exists a countable uniformly discrete metric space $M$ such that $\pna(M) = \lip(M)$ and, moreover, $\pna(M)$ is isometrically isomorphic to $\ell_\infty$. 
\end{proposition}

\begin{proof}
Define $M = \{0\} \cup \{p_{n}\}_{n=1}^{\infty},$
endowed with the following metric:
$$d(p,q):=\begin{cases}
0,\quad &\text{if $p=q$}\\
1,\quad &\text{if }p\neq q \text{ and }0\in\{p,q\}\\
2,\quad &\text{if }p=p_{n},\, q=p_{m} \text{ for some }n,m\in\bbn,\, n\neq m.
\end{cases}$$

Note that each Lipschitz function $f\in\lip(M)$ can be expressed as a sequence $a:=(a_n)_n$ by simply writing $a_n:=f(p_n)$ for each $n\in\bbn$. If the sequence $a$ is not bounded, then in particular there is a subsequence $(a_{n_k})_k$ such that $\lim_k |a_{n_k}|= \infty$. In this case, we get that $|S(f, 0, p_{n_k})|$ converges to $\infty$; hence $f$ is not Lipschitz. It follows that $a \in \ell_\infty$. Conversely, each element $a=(a_n)_n$ in $\ell_\infty$ induces a Lipschitz function $f_a$ by means of $f_a (p_n) := a_n$ for all $n\in\bbn$ and $f_a(0):=0$. 

We claim that $\|f_a \| = \|a\|$ and $f$ attains its pointwise norm at the point $0$. Let $p, q \in M$ with $p\neq q$ be given. 
\begin{enumerate}
\itemsep0.3em
\item If $p=0$ and $q=p_n$ for some $n\in\bbn$, we get
$|S(f_a, p, q)|=|a_n| \leq \|a \|.$

\item If there are $m,n\in\bbn$ with $m\neq n$ such that $p=p_n$ and $q=p_m$, then $$|S(f_a, p, q)|\leq \frac{|a_n|+|a_m|}{2}\leq \|a\|.$$
\end{enumerate}
Consequently, the mapping $a\mapsto f_a$ is an isometric isomorphism from $\ell_\infty$ onto $\lip(M)$, and each $f_a$ attains its pointwise norm at $0$.
\end{proof}

Notice that the metric space given in Proposition \ref{Countable-M-EllInfty} is not compact. In the following result, an example of a compact metric space $M$ for which every Lipschitz function attains its pointwise norm is presented. 



\begin{proposition}
There is an infinite compact metric space $M$ such that $\pna(M) = \lip(M)$.
\end{proposition}

\begin{proof}
Let $M = \{0\} \cup \{ 1/2^{n^2} \}_{n=1}^\infty \subseteq [0,1]$. Let $f \in \lip (M)$ with $\|f\|=1$ be given. For each $n \in \N$, take $p_n \neq q_n$ in $M$ with $p_n < q_n$ and $|S(f,p_n,q_n)| > 1 - 1/2^n$. 
Assume first that the sequence $\{q_n\}_n$ is decreasing and converges to $0$. Put $q_n = 1/2^{m_n^2}$ with $m_n \in \N$ for every $n \in \mathbb{N}$. Note that $p_n \leq 1/2^{(m_n+1)^2}$. Then 
\begin{align*}
|f(q_n)| \geq |f(q_n)-f(p_n)| - |f(p_n)| &\geq \left(1-\frac{1}{2^n}\right) \left( \frac{1}{2^{m_n^2}} - \frac{1}{2^{(m_n+1)^2}} \right) - \frac{1}{2^{(m_n+1)^2}} \\ 
&\geq \left( 1- \frac{2}{2^n} \right) \frac{1}{2^{m_n^2}}
\end{align*} 
since $m_n \geq n \geq 1$. This shows that $f$ attains its pointwise norm at $0$.

If it is not the case, then there is a subsequence of $\{p_n\}_n$ converging to some non-zero isolated point $a \in M$. This implies that $f$ attains its pointwise norm at the point $a$. 
\end{proof}

However, in the case when $M=[0,1]$, we show that there is a Lipschitz function on $M$ which does not attain its pointwise norm as follows.  


\begin{proposition}\label{thm:ell-infty-in-pna-01}
For $M=[0,1] \subseteq \mathbb{R}$, we have $\pna([0,1])\neq \lip([0,1])$, while the set $\pna([0,1])$ contains $\ell_\infty$ isometrically.
\end{proposition}

\begin{proof}
For each $n\in\bbn$, denote
$$f_n(x):=\max\left\{0,\, \frac{2^{2n-1}-1}{2^{n^2+1}} - \left| x-\frac{2^{2n-1}+1}{2^{n^2+1}} \right|\right\},\quad \text{for each }x\in [0,1].$$
We will show that the Lipschitz functions $\{f_n\}_{n=1}^{\infty}$ are isometrically equivalent to the canonical coordinate vectors $\{e_n\}_{n=1}^{\infty}\subset \ell_\infty$ in $\pna([0,1])$. Note that for each $n\in\bbn$, we have that 
\[
\{x\in [0,1]:\, f_n(x)\neq 0\}=\left] \frac{1}{2^{n^2}}, \frac{1}{2^{{(n-1)}^2}} \right[.
\]

Let $a=(a_1, a_2, \ldots)\in\ell_\infty$, and consider $f$ to be the pointwise limit of $\sum_{n=1}^{\infty} a_n f_n$.  Two scenarios can happen:
\begin{enumerate}
\item If there exists some $k_0\in\bbn$ such that $\|a\|=|a_{k_0}|$, then arguing as in the proof of \cite[Lemma 3.3]{DMQR23}, we conclude that $\|f\|=\|a\|$ and that $f$ strongly attains its norm at the pair of points $\left( \frac{1}{2^{n^2}},\, \frac{2^{2n-1}+1}{2^{n^2+1}} \right)\in\widetilde{M}$.
\item Otherwise, there exists a subsequence $(a_{n_k})_{k=1}^{\infty}$ of $(a_n)_{n=1}^{\infty}$ such that $\lim_{k\to\infty} |a_k|=\|a\|$. Arguing again as in \cite[Lemma 3.3]{DMQR23}, we know that for all $p\in [0,1]$,
$$\sup_{q\in [0,1]\backslash\{p\}} | S(f,p,q) | \leq \|a\|.$$
Moreover, note that 
$$\sup_{q\in ]0,1]} |S(f,0,q)| \geq \lim_{k\to\infty} |a_{n_k}|\frac{2^{2k+1}-1}{2^{2k+1}+1} = \|a\|,$$
which implies that $f$ has norm $\|a\|$ and it attains it pointwise at the point $0$.
\end{enumerate}
In both cases, $f\in\pna([0,1])$ with $\|f\| = \|a\|$. This proves that $\ell_\infty$ is isometrically contained in $\pna([0,1])$. 

Next, we construct a function $g \in \Lip([0,1])$ which does not attain its pointwise norm. Let $0<\eps, \eta <1$ be fixed and let
\begin{itemize}
\itemsep0.3em
\item $p_1 := (1,0)$;
\item $g_1$ be the line with slope $-(1-\eta)$ passing through $p_1$;
\item $q_1$ be the intersection point of the line $g_1$ and the line $y=\eps x$;
\item $h_1$ be the line with slope $1-\eta$ passing through $q_1$. 
\end{itemize}
Inductively, suppose we have defined the points $p_n$ and $q_n$, and the lines $g_n$ and $h_{n}$ for $n \in \mathbb{N}$. Then let
\begin{itemize}
\itemsep0.3em
\item $p_{n+1}$ be the intersection point of the line $h_n$ and the $x$-axis; 
\item $g_{n+1}$ be the line with slope $-(1-\eta^{n+1})$ passing through $p_{n+1}$;
\item $q_{n+1}$ be the intersection point of the line $g_{n+1}$ and the line $y=\eps x$; 
\item $h_{n+1}$ be the line with slope $1-{\eta^{n+1}}$ passing through $q_{n+1}$. 
\end{itemize}
Let $g \in \Lip ([0,1])$ be the function with $g(0)=0$ whose graph is the union of line segments $[p_n, q_n]$ and $[q_n, p_{n+1}]$ with $n \in \mathbb{N}$ in $\mathbb{R}^2$ (that is, the union of cones, see Figure \ref{figure}). Note that $g$ cannot belong to $\pna ([0,1])$. As a matter of fact, observe that $\|g\| = 1$ while 
\[
\sup_{q\in\bbr\backslash\{p\}} |S(g,p,q)| < 1 \text{ for any fixed } p \in (0,1] \,\, \text{ and } \sup_{q\in\bbr\backslash\{0\}} |S(g,0,q)| \leq \varepsilon.
\]
This completes the proof. \end{proof}
\vspace{0.3em}

\begin{center}
\begin{figure}[H]
\centering
\includegraphics[scale=7]{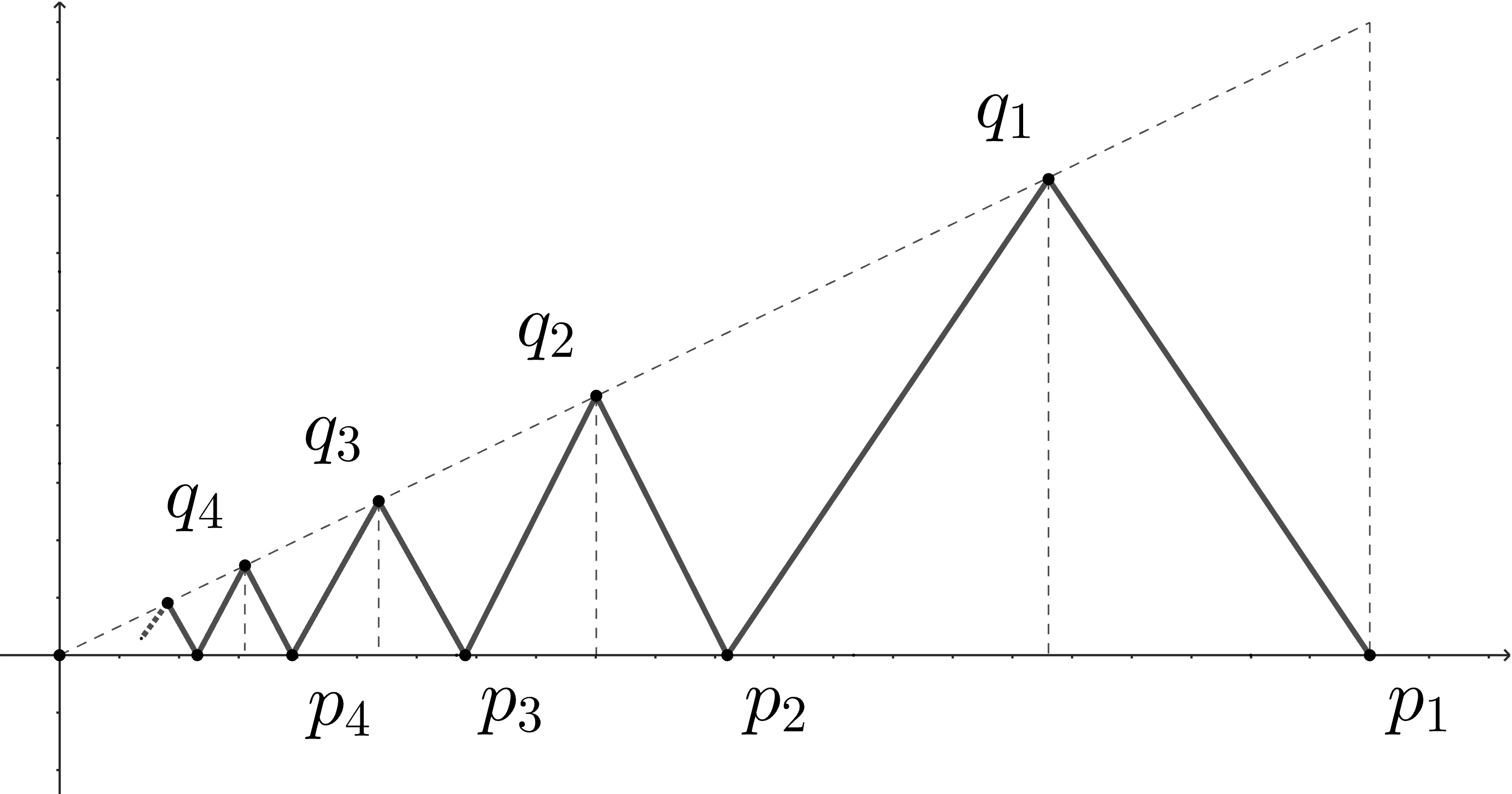}
\caption{The graph of the function $g \in \Lip([0,1])$ from Proposition \ref{thm:ell-infty-in-pna-01}}     \label{figure}
\end{figure}
\end{center}
\vspace{-2em}

We have seen some positive results about the embedding of $\ell_\infty$ in $\pna(M)$, but note that, as announced, this is not always possible, as there is a metric space $M$ such that $\pna(M)$ does not contain $\ell_1$ isometrically (in particular, it cannot contain $\ell_\infty$ either). In fact, such a space can be obtained by removing a single point from the metric space considered in the proof of Proposition \ref{Countable-M-EllInfty}, which shows that even a slight modification in the domain space can have drastic consequences in the spaceability properties of $\pna(M)$. Note that this result solves Question \ref{Q:PNA-Ell-Infty} in the negative.

\begin{theorem}\label{theorem:negative-example-ell1}
Let $M=\{p_n\}_{n=1}^{\infty}$ be a countable pointed metric space with $p_1=0$ and endowed with the usual discrete metric, $d(p,q)=1$ for all $(p,q)\in\widetilde{M}$. Then $\pna(M)$ cannot contain $\ell_1$ isometrically.
\end{theorem}

\begin{proof}

Suppose that $\pna(M)$ contains $\ell_1$ isometrically. Take $\{g_n\}_{n=1}^{\infty}\subset \pna(M)$ which corresponds to the canonical basis of $\ell_1$.
For each $s=(s_n)_n \in \{-1,1\}^\mathbb{N}$, let $f_s=\sum_{n=1}^{\infty} 2^{-n} s_n g_n$ and say $f_s$ attains its pointwise norm at $p_{k_s} \in M$ for some $k_s \in \mathbb{N}$. 
By considering $-f_s$ instead of $f_s$ if necessary, we may assume that 
\begin{equation*}
\sup_{\ell \neq k_s} \{ f_s (p_{k_s}) - f_s (p_\ell) \}= 1 \text{ for each } s \in \{-1,1\}^\mathbb{N}. 
\end{equation*} 
For a fixed $s \in \{-1,1\}^{\mathbb{N}}$, observe that 
\begin{equation}\label{eq:k_s}
1 = \sup_{\ell \neq k_s} \{ f_s (p_{k_s}) - f_s (p_\ell) \} \leq \sum_{n=1}^\infty \frac{s_n}{2^n} \Big( \sup_{\ell \neq k_s} \{g_n (p_{k_s})- q_n (p_\ell) \} \Big) \leq 1, 
\end{equation}
which implies that $ \sup_{\ell \neq k_s} \{g_n (p_{k_s})- g_n (p_\ell) \} = \sign (s_n)$ for every $n \in \mathbb{N}$.

Let $s=(s_n)_n$ and $t=(t_n)_n$ be two sequences of signs from $\{-1,1\}^{\bbn}$ such that $s\neq t$ and $s\neq -t$. Let us choose $n_1,n_2\in\bbn$ such that $s_{n_1}=t_{n_1}$ and $s_{n_2}= -t_{n_2}$. We claim that ${k_s} \neq {k_t}$. Observe from \eqref{eq:k_s} that 
\begin{itemize}
\itemsep0.3em
\item $\sup_{\ell \neq k_s} \{g_n (p_{k_s})- g_n (p_\ell) \} = \sign (s_n)$ for every $n \in \N$; 
\item $\sup_{\ell \neq k_t} \{g_n (p_{k_t})- g_n (p_\ell) \} = \sign (t_n)$ for every $n \in \N$. 
\end{itemize} 
Thus, if ${k_s} = {k_t}$, then $\sign(s_n)=\sign(t_n)$ for every $n \in \mathbb{N}$ which contradicts the choice of $n_1$ and $n_2$. 
It follows that the points $p_{k_s}$ and $p_{k_t}$ are distinct in $M$ whenever $s\neq t$ and $s\neq -t$; hence $M$ must be uncountable. This contradicts that $M$ is countable, so $\pna(M)$ cannot contain $\ell_1$ isometrically. 
\end{proof}

We refer to Section \ref{Section-PNA-Ell-1} for further results about the embeddability of $\ell_1$ in the sets of pointwise or strongly norm-attaining Lipschitz functions.





\section{Embedding of \texorpdfstring{$c_0$}{c0} in strongly norm-attaining Lipschitz functions}\label{Section-SNA-c0}

Before we provide our main result of the paper in Section \ref{Section-PNA-c0} which deals with the isometric embedding of $c_0$ in the set of pointwise norm-attaining Lipschitz functions, we first focus on some new results on the isometric embedding of $c_0$ in the set of strongly norm-attaining Lipschitz functions in this section. Some of the results in this section will be mentioned in Section \ref{Section-PNA-c0} and we will highlight some notable differences between strongly norm-attainment and pointwise norm-attainment. To this end, we begin with some positive results on the embeddings of $c_0$ in $\sna(M)$ and, on the other hand, find some classes of spaces where it is not possible to embed $c_0$ isometrically in $\sna(M)$ (extending \cite[Theorems 4.1 and 4.4]{DMQR23}). The rest of the section will be devoted to some related results and applications involving Lipschitz-free spaces. We show in particular that there exists a metric space $M$ such that $\na(\fm)$ does not contain $c_0$ isometrically, and that if $M$ is a subset of an $\bbr$-tree that contains all its branching points, then $c_0$ is isometrically contained in $\sna(M)$ (the necessary definitions will be given later).





Recall from \cite[Lemma 3.3]{DMQR23} that if for some infinite index set $\Gamma$ a metric space $M$ contains $\{p_\gamma, q_\gamma\}_{\gamma \in \Gamma} \subseteq \widetilde{M}$ such that 
\begin{equation*}\label{eq:lem3.3_dmqr23}
d(p_\alpha,q_\beta) \geq d(p_\alpha,q_\alpha)+d(p_\beta,q_\beta) \quad \text{for every } \alpha \neq \beta \in \Gamma, 
\end{equation*}
then $\sna(M)$ contains $c_0(\Gamma)$ isometrically. 
To the best of our knowledge, this result and the one about infinite non-uniformly discrete metric spaces (see \ref{SS}) stand as the only known positive results concerning the isometric embedding of $c_0$ into $\sna(M)$.
We will see that there are more cases, as the following results show. We first provide a version of \cite[Lemma 3.3]{DMQR23} for uniformly discrete spaces. For simplicity, from now on we shall use the following notation: given $p\in M$, let $R(p):=\inf\{d(p,q):\, q\in M\backslash\{p\}\}$.

\begin{proposition}\label{prop:c0_characterizationQ}
Let $M$ be an infinite uniformly discrete pointed metric space and $\Gamma$ be a nonempty index set. Suppose that $M$ contains distinct points $\{p_\gamma\}_{\gamma \in \Gamma}$. 
The following claims are equivalent: \begin{itemize}
\item[\textup{(a)}] $\sna(M)$ contains an isometric copy of $c_0(\Gamma)$ with functions $\{f_\gamma\}_{\gamma \in \Gamma}$ given by
\[
f_\gamma (p_\gamma) = a_\gamma \in \mathbb{R}, \quad f_\gamma (p) = 0 \, \text{ if } p \in M \setminus \{p_\gamma\} 
\]
that correspond to $\{e_\gamma\}_{\gamma \in \Gamma} \subseteq c_0 (\Gamma)$. 
\item[\textup{(b)}] For each $\gamma \in \Gamma$, there exists some $q_\gamma\in M\backslash \{p_\gamma\}$ such that $d(p_\gamma, q_\gamma)=R(p_\gamma)$ and the following holds
$$d(p_\alpha, p_\beta)\geq d(p_\alpha, q_\alpha) + d(p_\beta, q_\beta) \text{ for all } \alpha \neq \beta \in \Gamma. 
$$
\end{itemize}
\end{proposition}

\begin{proof} 
As the implication $(b) \Rightarrow (a)$ is proved in \cite[Lemma 3.3]{DMQR23}, we only prove $(a) \Rightarrow (b)$. By the definition of $f_\gamma$, we see that there exist some point $q_\gamma \in M\backslash\{p_\gamma \}$ such that $f_\gamma$ strongly attains its norm at $(p_\gamma, q_\gamma)$. Note that this implies that $d(p_\gamma, q_\gamma)=R(p_\gamma)$ and that $a_\gamma=\pm R(p_\gamma)$. Finally, consider $\alpha \neq \beta \in\Gamma$. Since the Lipschitz functons $f_\alpha\pm f_\beta$ must have norm $1$, we get in particular that 
$$\frac{R(p_\alpha)+R(p_\beta)}{d(p_\alpha, p_\beta)}=\frac{d(p_\alpha, q_\alpha)+d(p_\beta, q_\beta)}{d(p_\alpha, p_\beta)}\leq 1,$$
which proves the claim.
\end{proof}

The previous result (in fact \cite[Lemma 3.3]{DMQR23}) can be applied for instance if $M$ contains a sequence of distinct aligned points, that is, if there exists $\{p_n\}_n \subseteq M$ such that $d(p_n,p_{n+2}) = d(p_n,p_{n+1}) + d(p_{n+1},p_{n+2})$ for every $n \in \N$. An application of the following lemma will be shown later in Theorem \ref{prop:trees}.

\begin{lemma}\label{Lemma:Aligned}
Let $M$ be an infinite pointed metric space. If $M$ contains a sequence of distinct aligned points, then $c_0$ is isometrically contained in $\sna(M)$.
\end{lemma}

\begin{proof}
Consider an ordered sequence of distinct aligned points $\{r_n\}_{n=1}^{\infty}$ in $M$. Define the subsequence $\{p_n\}_{n=1}^{\infty}$ and the subset  $\{q_n:\, n\in\mathbb{N}\}$ of $\{r_n\}_{n=1}^{\infty}$ as follows: for each $n\in\mathbb{N}$, put $p_n=r_{2n-1}$ and let $q_n$ be the closest element to $p_n$ among $\{r_n\}_{n=1}^{\infty}$ (if there are 2 points at the same distance, choose the one with lowest index). Clearly, for all $n\in\mathbb{N}$, $q_n$ is either $r_{2n-2}$ or $r_{2n}$. Then for any different positive integers $n<m$, the following is satisfied
\[
d(p_n, p_m)=d(r_{2n-1}, r_{2m-1})\geq d(r_{2n-1}, r_{2n})+d(r_{2m-2}, r_{2m-1}) \geq d(p_n, q_n)+d(q_m, p_m).
\]
Therefore, we can apply \cite[Lemma 3.3]{DMQR23} to $\{(p_n, q_n)\}_{n \in \mathbb{N}} \subseteq \widetilde{M}$ and finish the proof.
\end{proof}


However, as we know by \cite[Theorems 4.1 and 4.4]{DMQR23}, there are many uniformly discrete spaces $M$ that do not meet the conditions from Proposition \ref{prop:c0_characterizationQ}. 
For instance, consider an infinite metric space $M$ with the usual discrete metric as in Example \ref{example:discrete-metric} or Theorem \ref{theorem:negative-example-ell1}, or the space in the following example.


\begin{example}\label{prop-first-weird-example}
Let $M=\{p_n\}_n$, where $p_n=(1+1/n)e_n\in c_0$ for $n \geq 2$ with $p_1=0$. Note that for all $n<m$, we have 
\[
d(p_n, p_m)=1+\frac{1}{n}=R(p_n).
\] 
It is not difficult to check that $M$ does not satisfy the assumption in Proposition \ref{prop:c0_characterizationQ}. 
\end{example}

Nevertheless, the following result shows that $c_0$ still can be embedded isometrically in $\sna(M)$ when $M$ is an infinite metric space with the usual discrete metric or the one in Example \ref{prop-first-weird-example}. 

\begin{theorem}\label{theorem:c0-condition-1-improved}
Let $M$ be an infinite uniformly discrete pointed metric space and let $\Gamma$ be an infinite index set. Suppose that $\{(p_\gamma, q_\gamma)\}_{\gamma \in \Gamma} \subseteq \widetilde{M}$. 
The following claims are equivalent: 
\begin{itemize}
\item[\textup{(a)}] $\sna(M)$ contains an isometric copy of $c_0(\Gamma)$ with functions $\{f_\gamma\}_{\gamma \in \Gamma}$ given by
\[
f_\gamma (p_\gamma)= {d(p_\gamma, q_\gamma)}/{2}, \quad f_\gamma (q_\gamma)= -{d(p_\gamma, q_\gamma)}/{2}, \quad  f_\gamma(p) = 0 \, \text{ if } p \in M \setminus \cup_\gamma \{p_\gamma, q_\gamma\},
\]
that correspond to $\{e_\gamma\}_{\gamma \in \Gamma} \subseteq c_0(\Gamma)$.
\item[\textup{(b)}] $R(p_\gamma), R(q_\gamma) \geq d(p_\gamma, q_\gamma)/2$ for each $\gamma \in \Gamma$, and the following assertion
\begin{equation}\label{eq:pnqn/2}
d(p_\alpha,q_\alpha) + d(p_\beta, q_\beta)\leq 2\min\bigl\{d(r, s):\, r \in\{p_\alpha, q_\alpha\},\, s \in\{p_\beta, q_\beta\}\bigr\}
\end{equation}
holds for all $\alpha \neq \beta \in \Gamma$. 
\end{itemize}
\end{theorem}

\begin{proof}
${(b) \Rightarrow (a)}$. Define the Lipschitz functions $f_\gamma$ as in (a). 
Let $a=\{a_\gamma\}_\gamma\in c_0(\Gamma)$ and $(p,q) \in \widetilde{M}$ be given, and put $f_a:=\sum_{\gamma \in \Gamma} a_\gamma f_\gamma$. 
If $p, q \not\in \{ p_\gamma, q_\gamma : \gamma \in \Gamma\}$, then it is clear that $S(f_a, p, q)=0$. Thus, we only consider the following: 
\begin{enumerate}
\itemsep0.3em
\item If only one of $p$ and $q$ belongs to $\{ p_\gamma, q_\gamma : \gamma \in \Gamma\}$, then $|S(f_a, p, q)|\leq \|a\|$ from the assumption on $R(p_\gamma)$ and $R(q_\gamma)$. 
\item If there is $\gamma \in \Gamma$ such that $\{p ,q\} =\{p_\gamma, q_\gamma\}$, then $|S(f_a, p, q)|=|a_\gamma|$.
\item Finally, if there are $(\alpha, \beta) \in \widetilde{\Gamma}$ such that $p \in\{p_\alpha, q_\alpha\}$ and $q \in\{p_\beta, q_\beta\}$, then
$$|S(f_a, p, q)|\leq \frac{|a_\alpha|d(p_\alpha, q_\alpha) + |a_\beta|d(p_\beta, q_\beta)}{2\min\bigl\{d(r, s):\, r \in\{p_\alpha, q_\alpha\bigr\},\, s \in\{p_\beta, q_\beta\}\}}\leq \|a\|.$$
\end{enumerate}
This proves that $f_a\in\sna(M)$ and that $\|f_a\|=a$.

${(a) \Rightarrow (b)}$. From the fact that $\|f_\gamma\|=1$ for each $\gamma \in \Gamma$, we get that both $R(p_\gamma), R(q_\gamma) \geq d(p_\gamma, q_\gamma)/2$. Next, let $\alpha \neq \beta \in \Gamma$, and note that $f_\alpha \pm f_\beta$ have norm $1$, from which \eqref{eq:pnqn/2} follows.\end{proof}


It is not difficult to see that a metric space in the following corollary satisfies the condition (b) in Theorem \ref{theorem:c0-condition-1-improved}. 

\begin{corollary}\label{theorem:c0-condition-1}
Let $M$ be an infinite pointed metric space that is uniformly discrete and such that there exist a collection of pairs $\{(p_\gamma,q_\gamma)\}_{\gamma \in\Gamma}\subset \widetilde{M}$, for some non-empty index set $\Gamma$, satisfying the following conditions:
\begin{itemize}
\itemsep0.3em
\item If $\alpha \neq \beta\in \Gamma$, $\{p_\alpha,q_\alpha\}\cap \{p_\beta,q_\beta\}=\emptyset$.
\item For all $\gamma\in\Gamma$, $R(p_\gamma)=R(q_\gamma)=d(p_\gamma, q_\gamma)$.
\end{itemize}
Then $\sna(M)$ contains $c_0(\Gamma)$ isometrically.
\end{corollary}

On the other hand, the following example shows that the condition (b) in Theorem \ref{theorem:c0-condition-1-improved} does not always hold.

\begin{example}\label{example:not_theorem:c0-condition-1-improved}
Let $M=\{0\}\cup \{p_n\}_{n=1}^{\infty}$, with metric given by $d(p_n, 0)=1+ 1/n$ for all $n$ and $d(p_n, p_m)=1+1/n+1/m$ for all $n\neq m\in \bbn$. Then, $M$ does not satisfy the condition (b) in Theorem \ref{theorem:c0-condition-1-improved}. 
\end{example} 

Note also that the metric space $M$ in Example \ref{example:not_theorem:c0-condition-1-improved} is not covered by Proposition \ref{prop:c0_characterizationQ}.
However, changing slightly the functions $f_n$ in the proof of Theorem \ref{theorem:c0-condition-1-improved}, we can still have an isometric copy of $c_0$ inside the set $\sna(M)$ for such an $M$ as the following result shows.





\begin{theorem}\label{thm:1+1/n+1/m}
Let $M$ be an infinite uniformly discrete pointed metric space and let $\Gamma$ be a nonempty index set. Suppose that $\{(p_\gamma, q_\gamma)\}_{\gamma \in \Gamma} \subseteq \widetilde{M}$. 
The following claims are equivalent: 
\begin{itemize}
\itemsep0.3em
\item[\textup{(a)}] $\sna(M)$ contains $c_0(\Gamma)$ isometrically with functions $\{f_\gamma\}_{\gamma \in \Gamma}$ given by
$$f_\gamma(x):=\begin{cases}
\frac{1}{2}({d(p_\gamma,q_\gamma)+ R(p_\gamma) -R(q_\gamma)}), \quad \text{if }p=p_\gamma\\
\frac{1}{2}(-d(p_\gamma,q_\gamma)+ R(p_\gamma) -R(q_\gamma)), \quad \text{if }p=q_\gamma\\
0,\quad \text{else}
\end{cases}$$
that correspond to $\{e_\gamma\}_{\gamma \in \Gamma} \subseteq c_0(\Gamma)$.
\item[\textup{(b)}] For every $\alpha \neq \beta \in \Gamma$, the following holds:
\vspace{0.3em}

\begin{enumerate}
\itemsep0.3em
\item[\textup{(1)}] $d(p_\alpha,q_\alpha) \leq R(p_\alpha) + R(q_\alpha)$; 
\item[\textup{(2)}] $d(p_\alpha,q_\alpha)+d(p_\beta,q_\beta) +R(p_\alpha) +R(p_\beta) - R(q_\alpha) - R(q_\beta) \leq 2 d(p_\alpha,p_\beta)$; 
\item[\textup{(3)}] $d(p_\alpha,q_\alpha)+d(p_\beta,q_\beta) +R(p_\alpha) - R(p_\beta)  - R(q_\alpha) +R(q_\beta) \leq 2 d(p_\alpha,q_\beta)$;
\item[\textup{(4)}] $d(p_\alpha,q_\alpha)+d(p_\beta,q_\beta) -R(p_\alpha) - R(p_\beta) + R(q_\alpha)  + R(q_\beta) \leq 2 d(q_\alpha,q_\beta)$. 
\end{enumerate} 
\end{itemize}
\end{theorem} 

\begin{proof}
${(b) \Rightarrow (a)}$. For each $\gamma \in \Gamma$, define the Lipschitz functions $f_\gamma$ as in (a). 
Let $\{a_\gamma\}_\gamma \in c_0(\Gamma)$ and $(p,q) \in \widetilde{M}$ be given, and consider $f = \sum_{\gamma\in\Gamma} a_\gamma f_\gamma$. If $p=p_\alpha$ and $q=p_\beta$ for some $\alpha \neq \beta \in \Gamma$, then 
\[
|S(f,p,q)| \leq \|a\| \left( \frac{d(p_\alpha,q_\alpha)+R(p_\alpha)-R(q_\alpha)+ d(p_\beta,q_\beta)+R(p_\beta)-R(q_\beta)}{2d(p_\alpha, p_\beta)} \right) \leq \|a\|
\]
by the assumption (2). Similarly, the cases when $(p,q)=(p_\alpha, q_\beta)$ and $(p,q)=(q_\alpha, q_\beta)$ can be treated in a similar way from the assumptions (3) and (4). Moreover, if $p=p_\gamma$ for some $\gamma \in \Gamma$ and $q \not\in \cup_\gamma \{p_\gamma,q_\gamma\}$, then $|S(f,p,q)| \leq 1$ again by the assumption (1). 
Finally, note that for $\gamma \in \Gamma$,  we have $S(f, p_\gamma, q_\gamma) = a_\gamma$. This completes the proof.

${(a) \Rightarrow (b)}$. Since each $f_\gamma$ has norm $1$, the inequality (1) is obtained. Next, for $\alpha \neq \beta \in \Gamma$, the fact that $f_\alpha \pm f_\beta$ have norm $1$ implies the inequalities (2)--(4). 
\end{proof} 

It is routine to check that the metric space $M$ in Example \ref{example:not_theorem:c0-condition-1-improved} satisfies the assumptions in Theorem \ref{thm:1+1/n+1/m}. 


Now, we will present some negative results and applications. 
To this end, recall the Lipschitz-free space over a metric space $M$: the space $\calf(M)=\overline{\operatorname{span}}\{\delta_x:\, x\in M\}$ formed by the evaluation mappings $\delta_x\in \lip(M)^*$ is called the \textit{Lipschitz-free space} (also known as Arens-Eells space or transportation cost space). Given $(p,q) \in \widetilde{M}$, the elements $m_{p,q}:=\frac{\delta_p - \delta_q}{d(p,q)}$ are called the \textit{molecules} of $\fm$ and it is well known that the closed convex hull of molecules is the closed unit ball of $\calf (M)$. One important universal property of Lipschitz-free space is linearization: given a Lipschitz function $f \in \lip (M,Y)$, there is a unique bounded linear operator $T_f \in \call (X,Y)$ with $\|T_f\| = \|f\|$ satisfying that $T_f (\delta_x) = f(x)$ for every $x \in M$.  
Now, it is easy to see that $f\in\lip(M,Y)$ strongly attains its norm at a pair $(p,q)$ if and only if the corresponding bounded linear operator $T_f \in \call(\fm, Y)$ attains its norm at the molecule $m_{p,q}$. For a solid background in Lipschitz functions and Lipschitz-free spaces, we refer the reader to \cite{Godefroy15, Weaver18}.

We start with a new tool to find negative examples of the isometric $c_0$-embedding in $\sna(M)$, which was found by Andr\'es Quilis. The authors are thankful to him for allowing us to include his proof of the following result. 

\begin{proposition}\label{prop:Quilis-Negative-Tool-Free}
If $\sna(M)$ contains $c_0(\Gamma)$ isometrically, then $\mathcal{F}(M)$ contains a $1$-complemented copy of $\ell_1 (\Gamma)$ isometrically.
\end{proposition}

\begin{proof}
Suppose that $\{f_\gamma\}_{\gamma \in \Gamma} \subseteq \sna(M)$ corresponds to $\{e_\gamma\}_{\gamma \in \Gamma} \subseteq c_0(\Gamma)$ isometrically. Say $|S(f_\gamma, p_\gamma,q_\gamma)| = 1$ for some $(p_\gamma, q_\gamma) \in \widetilde{M}$. 
Let $a:= \{a_\gamma\}_{\gamma \in\Gamma} \in \ell_1 (\Gamma)$ and $F$ be a finite subset of $\Gamma$. Then 
\[
\sum_{\gamma \in F} | a_\gamma| = \sum_{\gamma \in F} \| a_\gamma m_{p_\gamma, q_\gamma} \| \leq \| a \|. 
\]
This shows that $\mu = \sum_{\gamma \in \Gamma} a_\gamma m_{p_\gamma, q_\gamma}$ is well-defined, and $\|\mu\| \leq \|a\|$. To obtain the reverse inequality, observe from \cite[Lemma 3.1]{DMQR23} that $S(f_\alpha, p_\beta, q_\beta) = 0$ whenever $\alpha \neq \beta \in \Gamma$. Given a finite set $F$ of $\Gamma$, let us consider $c_\gamma = \sign (S(f_\gamma, p_\gamma, q_\gamma)) \sign (a_\gamma)$ for each $\gamma \in F$. Put $f = \sum_{\gamma \in F} c_\gamma f_\gamma \in \sna(M)$. Note that $\|f\| \leq 1$, and 
\[
\| \mu \| \geq  | \langle \mu,\, f \rangle | =  \sum_{\gamma \in F} |a_\gamma|,
\]     
which shows that $\|\mu \| = \|a\|$; hence $\mathcal{F}(M)$ contains a space $Y$ isometrically isomorphic to $\ell_1 (\Gamma)$, where the molecules $\{m_{p_\gamma, q_\gamma}\}_{\gamma \in \Gamma}$ are isometrically equivalent to the canonical vectors $\{e^*_\gamma\}_{\gamma \in \Gamma}$ of $\ell_1(\Gamma)$. 

Finally, we claim that $Y$ is $1$-complemented in $\fm$. To this end, define $P:\fm\rightarrow \fm$ as
$$P(\mu):=\sum_{\gamma\in \Gamma} \langle \mu,\, f_\gamma\rangle m_{p_\gamma, q_\gamma},\quad \text{for all }\mu\in\fm.$$
By definition, $P$ is a projection. Moreover, given $\mu\in \fm$,  let $\theta_\gamma:=\sign(\langle \mu,\, f_\gamma \rangle)$ for each $\gamma\in\Gamma$. Then 
\begin{align*}
\sum_{\gamma\in \Gamma} |\langle \mu,\, f_\gamma\rangle | &= \sup\left\{\left\langle \mu,\, \sum_{\gamma\in F} \theta_\gamma f_\gamma\right\rangle:\, F\subseteq \Gamma \text{ finite}\right\} \\ 
&\leq \|\mu\| \sup\left\{\left\|\sum_{\gamma\in F} \theta_\gamma f_\gamma\right\|:\, F\subseteq \Gamma \text{ finite}\right\}=\|\mu\|,
\end{align*}
where the last equality holds since $\{f_\gamma\}_{\gamma \in \Gamma}$ corresponds to $\{e_\gamma\}_{\gamma\in\Gamma} \subseteq c_0(\Gamma)$. \end{proof}



In \cite[Theorem 2.1]{OO20}, it is characterized when $\mathcal{F}(M)$ contains $\ell_1$ isometrically in terms of perfect matchings with minimum weights (for its definition, see \cite{OO20}). Using the above result combined with \cite[Theorem 2.1]{OO20}, we directly obtain new examples of metric spaces $M$ where $\sna(M)$ does not contain $c_0$ isometrically (see \cite[Remark 10]{CJ17} and \cite[Example 2.6]{OO20}). This approach, for instance, provides a different proof of \cite[Theorem 4.1]{DMQR23} by \cite[Example 2.6]{OO20}. 

Using ideas partially inspired by \cite[Example 2.6]{OO20}, we will now present a consequence of the previous result which provides a wide class of uniformly discrete metric spaces $M$ such that $c_0$ is not isometrically contained in $\sna(M)$. Every bounded metric space from \cite[Remark 10]{CJ17} is a particular case of this class.


\begin{theorem}\label{thm:counterexamples-c0-SNA-2}
Suppose that $M=\{p_n\}_{n=1}^{\infty}$ is a uniformly discrete pointed metric space endowed with the metric 
$$d(p_n, p_m)=L+\phi(n,m),\quad \text{for all }(n,m) \in\widetilde{\bbn},$$
for some $L>0$ and some symmetric function $\phi : \mathbb{N}^2 \rightarrow \mathbb{R}$ with the following properties:
\begin{enumerate}
\itemsep0.3em
\item[\textup{(i)}]
 For all $n\in\bbn$, the limit $\psi(n):=\lim_{m\to\infty} \phi(n, m)$ exists.
\item[\textup{(ii)}]
 For all $(n,m)\in\widetilde{\bbn}$, $\phi(n,m)>\psi(n)+\psi(m)$.
\item[\textup{(iii)}]
 $\liminf \{ \phi(n,m) : (n,m) \in \widetilde{\mathbb{N}}\} \geq 0$. 
\end{enumerate}
Then $\fm$ does not contain an isometric copy of $\ell_1$. Consequently, for every metric space  $M_0\subset M$, the set $\sna(M_0)$ does not contain $c_0$ isometrically.
\end{theorem}

\begin{proof}
Assume that $\ell_1$ is isometrically contained in $\fm$. Then, by \cite[Theorem 2.1]{OO20}, there exists a sequence of pairs of distinct points $\{(u_n, v_n)\}_{n=1}^{\infty}$ in $\widetilde{M}$ such that for all $N\in\bbn$ and every permutation $\sigma:\{1,\ldots,N\}\rightarrow \{1,\ldots,N\}$,
$$\sum_{n=1}^N d(u_n, v_n) \leq \sum_{n=1}^N d(u_n, v_{\sigma(n)}).$$
Note that if such a sequence exists, every subsequence of it will also satisfy that condition. In particular, up to relabeling and taking subsequences if needed, there exist sequences of integers $\{k_n\}_n$ and $\{j_n\}_n$ such that for all $n\in\bbn$, the following holds:
\begin{itemize}
\item $k_n < j_n < k_{n+1}$ and $u_n=p_{k_n}$, $v_n=p_{j_n}.$
\end{itemize}
Then, for all $n<m$, 
$$\phi(k_n, j_n) + \phi(k_m, j_m) \leq \phi(k_n, j_m) + \phi(k_m, j_n).
$$ 
It follows that by letting $m \rightarrow \infty$,  
\[
\phi(k_n, j_n) + \liminf \{ \phi(i,j) : (i,j) \in \widetilde{\mathbb{N}} \} \leq \psi(k_n) + \psi(j_n). 
\]
From condition (iv), the left-hand side is greater than or equal to $\phi(k_n, j_n)$. This contradicts the condition (iii). Thus, we conclude that $\fm$ does not contain an isometric copy of $\ell_1$. 

Finally, the above argument combined with Proposition \ref{prop:Quilis-Negative-Tool-Free} shows that $c_0$ is not isometrically contained in $\sna(M_0)$ for any metric space $M_0 \subseteq M$. 
\end{proof}

As a consequence of Theorem \ref{thm:counterexamples-c0-SNA-2}, we have the following corollary which generalizes \cite[Theorem 4.1]{DMQR23}.

\begin{corollary}\label{Thm-Negative-c0-SNA}
Let $M = \{ p_n\}_n$ be a pointed metric space satisfying that 
\begin{enumerate}
\itemsep0.3em
\item[\textup{(i)}] $d(p_n,p_m)$ is decreasing with respect to $n$ and $m$, and $\lim_m d(p_n,p_m) = L>0$ for each $n \in \N$. 
\item[\textup{(ii)}] $d(p_n,p_m)>L$ for every $n \neq m$. 
\end{enumerate} 
Then $c_0$ is not isometrically contained in $\sna (M)$. 
\end{corollary} 

However, the proof of Theorem \ref{thm:counterexamples-c0-SNA-2} depends on the result \cite[Theorem 2.1]{OO20}. Let us present an alternative and elementary proof of Corollary \ref{Thm-Negative-c0-SNA} inspired by the argument used in the proof of \cite[Theorem 4.1]{DMQR23}. 

As we make use of Ramsey's theorem in the proof (and later in Section \ref{Section-PNA-c0} as well), let us fix some notation. For any set $A$ and any $n\in\bbn$, let $A^{[n]}$ denote the set of all subsets of $A$ with exactly $n$ elements. Recall that Ramsey's Theorem states that given any infinite set $A$ and any finite partition $\{B_1,\dots,B_k\}$ ($k\in\bbn$) of the set $A^{[n]}$, there exists an infinite subset $S$ of $A$ and a number $i\in\{1,\dots k\}$ such that $S^{[n]}$ is contained in $B_i$ (see, for instance, \cite[Proposition 6.4]{FHHMZ11}).

\begin{proof}[\mbox{Alternative proof of Corollary \ref{Thm-Negative-c0-SNA}}]
Assume that there is $\{f_n\}_n \subseteq \sna (M)$ isometric to $c_0$-basis. For each $n \in \mathbb{N}$, let $u_n \neq v_n \in M$ be such that $|f_n(u_n) - f_n(v_n) | = d(u_n, v_n)$. Consider 
\begin{align*}
&A = \{ \{n,m\} \in \mathbb{N}^{[2]} : \{u_n,v_n\} \cap \{u_m, v_m\} = \emptyset\}; \\
&B_1 = \{ \{n,m\} \in \mathbb{N}^{[2]} : u_n = u_m \}, \,\, B_2 = \{ \{n,m\} \in \mathbb{N}^{[2]} : v_n = v_m \};  \\ 
&B_3 = \{ \{n,m\} \in \mathbb{N}^{[2]} : u_n = v_m \text{ or } u_m = v_n\}. 
\end{align*} 
By Ramsey's theorem, there exists $C \in \{A, B_1, B_2, B_3\}$ and an infinite set $S \subseteq \mathbb{N}$ such that $S^{[2]} \subseteq C$. 

\textit{Case I}. $C = A$. Let us restrict ourselves to $S$. Fix $n_0 \in S$, and let $\eps_0 := d(u_{n_0},v_{n_0})-L > 0$. By \cite[Lemma 3.2]{DMQR23} and our assumption, there exists $m_0 \in \mathbb{N} \setminus \{n_0\}$ such that 
\begin{itemize}
\itemsep0.3em
\item $\max \{ |f_{m_0} (u_{n_0})| , |f_{m_0} (v_{n_0}) | \} <\eps_0/6$;
\item $\max \{ d(u_{n_0}, u_{m_0}), d(u_{n_0}, v_{m_0}), d(u_{m_0}, v_{n_0}), d(v_{n_0}, v_{m_0}) \} < L + \eps_0/6$. 
\end{itemize} 
Note from \cite[Lemma 3.1]{DMQR23} that $f_{n_0} (u_{m_0}) = f_{n_0} (v_{m_0}) = C_{m_0}$ for some $C_{m_0} \in \mathbb{R}$. By relabelling if necessary, we may assume that 
\begin{equation}\label{eq:relabelling_assumption} 
|f_{n_0} (u_{n_0})-C_{m_0}| \geq |f_{n_0} (v_{n_0})-C_{m_0}|\, \text{ and } \, |f_{m_0} (u_{m_0})| \geq |f_{m_0} (v_{m_0}) |. 
\end{equation} 
By triangle inequality, we have 
\begin{itemize}
\itemsep0.3em
\item $|f_{n_0} (u_{n_0})-C_{m_0}| + |f_{n_0} (v_{n_0})-C_{m_0}| \geq d(u_{n_0}, v_{n_0})$; 
\item $|f_{m_0}(u_{m_0})| + |f_{m_0}(v_{m_0})|  \geq d(u_{m_0}, v_{m_0})$. 
\end{itemize} 
These two inequalities, with the assumption \eqref{eq:relabelling_assumption}, imply that 
\begin{itemize}
\itemsep0.3em
\item $|f_{n_0} (u_{n_0})-C_{m_0}|  \geq d(u_{n_0}, v_{n_0})/2$ \, and \, $|f_{m_0}(u_{m_0})|   \geq d(u_{m_0}, v_{m_0})/2$.
\end{itemize} 
Set $\delta := \sign ( { f_{m_0} (u_{m_0} ) } ) $. We distinguish two cases depending on the sign of $f_{n_0}(u_{n_0}) - C_{m_0}$. First, assume that $f_{n_0} (u_{n_0}) < C_{m_0}$. Then letting $f := f_{n_0} + \delta f_{m_0}$, we have 
\begin{align*}
|f(u_{n_0})-f (u_{m_0})| &\geq - f_{n_0} (u_{n_0}) - \delta f_{m_0} (u_{n_0}) + f_{n_0} (u_{m_0}) + \delta f_{m_0} (u_{m_0}) \\ 
&\geq \frac{1}{2} d(u_{n_0} , v_{n_0}) - \delta f_{m_0} (u_{n_0}) + |f_{m_0}( u_{m_0}) | \\
&\geq \frac{1}{2} d(u_{n_0} , v_{n_0}) - |f_{m_0} (u_{n_0})| + \frac{1}{2} d(u_{m_0}, v_{m_0}) \\ 
&\geq \frac{1}{2} (L+\eps_0) - \frac{\eps_0}{6} + \frac{1}{2} L = L + \frac{\eps_0}{3}  > L + \frac{\eps_0}{6} > d(u_{n_0}, u_{m_0}),
\end{align*} 
which shows that $\|f\| > 1$; so it is a contradiction. Second, assume that $f_{n_0} (u_{n_0})>C_{m_0}$. Then considering $g:= f_{n_0} - \delta f_{m_0}$ and arguing similarly, we can deduce that the norm of $g$ is greater than $1$, which is again a contradiction.

\textit{Case II}. $C=B_1$. Let us restrict ourselves to $S$. Note that $v_n \neq v_m$ for every $n\neq m \in S$. Put $\eps_n = d(u_n,v_n) - L$ and fine $n_0 \in S$ so that $\eps_{n_0} < \frac{1}{5} L$. As $u_n$'s are all the same, by \cite[Lemma 3.2]{DMQR23} we can find $n_1 \neq m_1 \in S$ such that 
\begin{equation}\label{eq:n1m1}
|f_{n_1}(u_{n_1})| < \eps_{n_0}, |f_{m_1}(u_{m_1})| < \eps_{n_0}, \text{ and } d(v_{n_1}, v_{m_1})<L+\eps_{n_0}. 
\end{equation} 
Then by triangle inequality, 
\begin{itemize}
\itemsep0.3em
\item $|f_{n_1}(v_{n_1})| > L -\eps_{n_0}$ \, and \, $|f_{m_1}(v_{m_1})| > L -\eps_{n_0}$.
\end{itemize} 
By \cite[Lemma 3.1]{DMQR23} again, notice that 
\[
|f_{n_1}(v_{m_1})| < \eps_{n_0} \, \text{ and } \, |f_{m_1}( v_{n_1})| < \eps_{n_0}. 
\]
By changing the signs if necessary, we may assume that $f_{n_1} ( v_{n_1}), f_{m_1} (v_{m_1}) >0$. Consider $f := f_{n_1}-f_{m_1}$. Then  
\begin{align*}
|f(v_{n_1}) - f(v_{m_1})| &\geq f_{n_1} (v_{n_1}) + f_{m_1} (v_{m_1}) - f_{m_1} (v_{n_1}) - f_{n_1} (v_{m_1}) \\
&> 2(L-\eps_{n_0}) - 2\eps_{n_0} > L + \eps_{n_0} > d(v_{n_1}, v_{m_1}). 
\end{align*} 
This shows that the norm of $f$ is greater than $1$, which is a contradiction. 

\textit{Case III}. $C= B_2$. As a matter of fact, this case can be solved from Case II by changing the role of $u_n$ and $v_n$. 

\textit{Case IV}. $C= B_3$. Arguing similarly as in the last step of the proof of \cite[Theorem 4.1]{DMQR23}, we can show that this case forces the cardinality of $S$ to be at most $3$. This contradicts that $S$ is an infinite subset of $\mathbb{N}$.

Consequently, we obtained a contradiction in any case. This shows that there exists no $c_0$-basis in $\sna (M)$. 
\end{proof}

Let us mention that the results in the next section (see Theorem \ref{theorem:bounded-counterexample2}, Theorem \ref{theorem:unbounded-counterexample} and Theorem \ref{Main-Theorem-c0-PNA}) can be used to observe that if $M$ is any metric space satisfying the conditions in Theorem \ref{thm:counterexamples-c0-SNA-2} or $M$ is any of the metric spaces considered in \cite[Remark 10]{CJ17} and \cite[Theorem 4.4]{DMQR23}, then there exists a metric space $M_0\subset M$ such that $\pna(M_0)$ contains an isometric copy of $c_0$ (and in fact, in every metric space from \cite[Remark 10]{CJ17}, \cite[Theorem 4.4]{DMQR23}, and Corollary \ref{Thm-Negative-c0-SNA}, the metric space $M_0$ can be chosen to be $M$).


Before we closed this section, let us observe some related results involving Lipschitz-free spaces. First, we present results on the embedding of $c_0$ in $\na(\fm)$. Recall that $\sna(M)$ is isometrically contained in $\na(\fm)$, which leads us to the following natural question.

\begin{question}\label{Q-na-fm-c0}
Does $\na(\fm)$ isometrically contains $c_0$ for any infinite metric space $M$?
\end{question}

Our goal is to provide an answer to Question \ref{Q-na-fm-c0}. It is shown in \cite[Theorem 4.2]{DMQR23} that $\sna(M)$ contains $c_0$ isometrically if $M$ is not uniformly discrete; hence so does $\naf$. However, we will show that there is a (uniformly discrete) metric space $M$ such that $\na (\fm)$ does not contain $c_0$ isometrically. This claim will follow from a general result on the convex integral of molecules. For the detailed explanation of notations, we refer the reader to \cite{APSpre}.

Let us denote by $C_b (\Omega)$ the space of all bounded continuous real-valued functions on a Hausdorff space $\Omega$. Let $M$ be a complete pointed metric space. 
Recall that the \textit{de Leeuw transform} is a linear isometry $\Phi : \lip(M) \rightarrow C_b (\widetilde{M}) = C(\beta \widetilde{M})$ given by 
\[
(\Phi f)(p,q) = S(f,p,q)
\]
for each $(p,q) \in \widetilde{M}$, where $\beta \widetilde{M}$ denotes the Stone-\v{C}ech compactification of $\widetilde{M}$. Note that $\widetilde{M}$ is a Borel set of $\beta\widetilde{M}$ when $M$ is complete (see, for instance, \cite[Proposition 2.6]{APSpre} and \cite[Theorem 24.12, Theorem 24.13]{Willard70}). 
An element $m \in \fm$ is called a \textit{convex integral of molecules} if it can be written as $\Phi^* \mu$ for some $\mu \in \mathcal{M}(\beta\widetilde{M})$ satisfying 
\begin{enumerate}[label=(C\arabic*)]
\itemsep0.3em
\item $\mu \geq 0$ and $\|m\| = \|\mu\|$; \label{C1}
\item $\mu$ is concentrated on $\widetilde{M}$, i.e., $\mu \vert_{\widetilde{M}} = \mu$. \label{C2}
\end{enumerate}

\begin{proposition}\label{prop:convex_integral}
Let $M$ be a complete pointed metric space. Suppose that every element of $\fm$ is a convex integral of molecules. Then $f \in \sna(M)$ whenever $T_f \in \na(\fm)$. 
\end{proposition}

\begin{proof}
Let $f \in \lip (M)$ and suppose that $|T_f (m)| = \|T_f\|$ for some $m\in \fm$ with $\|m\|=1$. As $m$ is a convex integral of molecules, there exists $\mu \in \mathcal{M}(\beta\widetilde{M})$ satisfying \ref{C1} and \ref{C2}. By \cite[Proposition 2.6]{APSpre}, we have 
\begin{align*}
\|T_f\| = |T_f (m)| = \langle \Phi^* \mu, T_f \rangle &= \int_{\widetilde{M}} \langle m_{p,q}, T_f \rangle\,d\mu(p,q) \\
&\leq \int_{\widetilde{M}} \left|S(f,p,q) \right| \,d|\mu|(p,q) \leq \|f\| \| \mu\| = \|f\|. 
\end{align*}
This implies that $|S(f,p,q)| = \|f\|$ almost everywhere, in particular, $f \in \sna(M)$.
\end{proof}

Recall that a metric space $M$ is called \textit{radially discrete} if there exists some $\alpha > 0$ such that $d(p,q)\geq \alpha d(p, 0)$ for every $(p,q)\in\widetilde{M}$. By \cite[Proposition 3.3]{APSpre} and \cite[Corollary 3.4]{APSpre}, every element of $\fm$ is a convex series of molecules if either $\fm$ is isometric to $\ell_1(\Gamma)$ for some index set $\Gamma$ or $M$ is radially discrete. If $M$ is radially discrete and the distinguished point $0$ is isolated, then $M$ is called \textit{radially uniformly discrete}. Note that all bounded uniformly discrete metric spaces are radially uniformly discrete. Thus, we have the following direct consequence of Proposition \ref{prop:convex_integral}.

\begin{corollary}\label{cor:convex_integral}
Let $M$ be a complete pointed metric space. Suppose that either $M$ is radially discrete or $\fm$ is isometric to $\ell_1 (\Gamma)$ for some index set $\Gamma$. Then $f \in \sna(M)$ whenever $T_f \in \na(\fm)$. 
\end{corollary}

This leads us to a negative answer to Question \ref{Q-na-fm-c0}.

\begin{corollary}\label{cor:convex_integral2}
There exists an infinite metric space $M$ such that $\na (\F (M))$ does not contain $c_0$ isometrically.
\end{corollary}

\begin{proof}
Consider any radially discrete metric space $M$ for which $\sna(M)$ does not contain an isometric copy of $c_0$ (for a concrete example, see \cite[Theroem 4.1, Theorem 4.4]{DMQR23}). Then Corollary \ref{cor:convex_integral} shows that $\na (\F (M))$ cannot contain an isometric copy of $c_0$. 
\end{proof}

In Proposition \ref{prop:convex_integral} we have shown that if $M$ is a complete metric space such that every element of $B_{\fm}$ is a convex integral of molecules, then $T_f\in\na(\fm)$ if and only if $f\in\sna(M)$. Therefore, for those $M$ satisfying that $\fm$ is isometric to $\ell_1(\Gamma)$, we have
\begin{equation}\label{eq:l_1(gamma)}
\sna(M)=\na(\fm) = \na (\ell_1(\Gamma)),
\end{equation}
under the isometric equivalence between $\lip(M)$ and $\call(\calf(M))$. It is clear that $c_0 (\Gamma)$ is isometrically contained into $\na (\ell_1(\Gamma))$ since $c_0(\Gamma)$ is an isometric predual of $\ell_1 (\Gamma)$. Thus we have from \eqref{eq:l_1(gamma)} the following for complete metric spaces: 
\begin{enumerate}[label=($\dagger$)]
\item \textit{$\sna(M)$ contains  $c_0(\Gamma)$ isometrically  when $\fm$  is isometric to  $\ell_1(\Gamma)$}. \label{eq:rmk_tree}
\end{enumerate}

Note from \cite[Theorem 3.2]{Godard} that if $M$ is a subset of an $\bbr$-tree $(T,d)$ such that $\overline{M}$ contains all branching points of $T$, then $\mathcal{F}(M)$ is isometric to $L_1 (\mu_M)$ for some measure $\mu_M$. Recall that an \textit{$\bbr$-tree} is a metric space $(T,d)$ that is geodesic and satisfies the \textit{$4$-point conditions}:
\[
d(p,q)+d(r,s) \leq \max \{ d(p,s)+d(q,r), d(q,s)+d(p,r)\}, \, \text{ for all } p,q,r,s \in T.
\]
A point $b \in T$ is said to be a \textit{branching point} if $T\setminus \{b\}$ has at least three connected components.
Moreover, it was characterized in \cite[Theorem 5]{DKP16} that $\fm$ is isometric to some $\ell_1(\Gamma)$ if and only if $M$ is a negligible subset of an $\bbr$-tree of cardinality $|\Gamma|-1$ which contains all the branching points (see \cite{DKP16} for the definition of negligible sets). 
For more details on $\bbr$-trees, see \cite{DKP16, Godard}.

As an application of our previous results, we observe the following which provides, 
This approach, combined with \cite[Theorem 5]{DKP16}, provides a more general conclusion than \ref{eq:rmk_tree} for countable index set $\Gamma$.

\begin{theorem}\label{prop:trees}
For any infinite pointed metric space $M$ that is a subset of an $\bbr$-tree and contains all the branching points, the set $\sna(M)$ contains $c_0$ isometrically.
\end{theorem}

\begin{proof}
By the fact \ref{SS}, it is enough to assume that $M$ is uniformly discrete. Let $\conv(M)$ denote a minimal $\bbr$-tree containing $M$. For every two points $p,q\in M$, if $[p,q]$ denotes the unique minimal-length arc connecting $p$ and $q$, let $[p,q]_M:=[p,q]\cap M$. 

We first claim that for any $p\in M$ and every connected component $C$ of $\conv(M)\backslash\{p\}$, there exists a point $q\in C$ such that $[p,q]_M=\{p,q\}$. Indeed, this is clear if $C\cap M$ is finite, so suppose that $C\cap M$ is infinite. 
Consider a sequence of points $\{r_n\}_{n=1}^{\infty}\subset C\cap M$ such that $d(p, r_n)$ converges to $\inf\{d(p, r):\, r\in C\cap M\}$, which is a positive number as $M$ is uniformly discrete. 
Note that for every two points $x,y\in C\cap M$, there exists a point $x\wedge y \in [p, x]\cap [p, y]\backslash \{p\}$ such that $[p, x]\cap [p, y]=[p, x\wedge y]$ and $x\wedge y\in [x,y]$. If $x \wedge y$ is either $x$ or $y$, then it is obvious that $x \wedge y \in C \cap M$. If this is not the case, then $x \wedge y$ would be a branching point. As $M$ contains all branching points, we obtain that $x \wedge y \in C \cap M$. Now, consider the sequence of points $\{s_n\}_{n=1}^{\infty} \subseteq C \cap M$ inductively given as follows: $s_1=r_1$, and given $s_1, \ldots, s_n$ for some $n\in\bbn$, put $s_{n+1}:=s_n\wedge r_{n+1}$. It is clear that $d(p, s_n)$ also converges to $\inf\{d(p, r):\, r\in C\cap M\}$. Since the points $\{s_n\}_{n=1}^{\infty}$ are aligned by construction and $M$ is uniformly discrete, we deduce that $\{s_n\}_{n=1}^{\infty}$ must be eventually constant.
This proves that there exists some $q \in C\cap M$ so that $[p,q]_M=\{p,q\}$.

To finish the proof, we shall make use of Proposition \ref{prop:c0_characterizationQ} and Lemma \ref{Lemma:Aligned}. Consider the following two cases.
\begin{enumerate}
\itemsep0.3em
\item Suppose that there exists $p\in M$ such that the set $\conv(M)\backslash \{p\}$ has an infinite number of connected components. Then applying the argument above, we can find a sequence of distinct points $\{p_n\}_{n=1}^{\infty}\subset M$ satisfying that $[p_n, q]_M=\{p_n, q\}$ for each $n\in\mathbb{N}$. Now,  Proposition \ref{prop:c0_characterizationQ} can be applied by considering $q_n:=q$ for all $n\in\bbn$.
\item Suppose that for all $p\in M$, the set $\conv(M)\backslash \{p\}$ has a finite number of connected components, $C_{p,1}, \ldots, C_{p,n_p}$ with  $n_p\in\bbn$. Fix a point $p\in M$ and let $C_{p,1}, \ldots, C_{p,n_p}$ be the connected components of $\conv(M)\backslash \{p\}$. 
Now, we construct a sequence of aligned points of $M$ starting on $p$. First, take $1\leq k_1\leq n_p$ such that $C_{p,k_1}\cap M$ is infinite, and let $q_1\in C_{p,k_1}$ such that $[p, q_1]_M=\{p, q_1\}$. If $C_{q_1,1}, \ldots, C_{q_1,n_{q_1}}$ are the connected components of $C_{p,k_1}\backslash \{q_1\}$, then there is $1\leq k_2\leq n_{q_1}$ such that $C_{q_1,k_2}\cap M$ is infinite. Take $q_2\in C_{q_1, k_2} \cap M$ such that $[q_1, q_2]_M=\{q_1, q_2\}$. Inductively, given $q_1, \ldots, q_j \in M$ for some $j\in\bbn$, if $C_{q_j,1}, \ldots, C_{q_j,n_{q_j}}$ are the connected components of $C_{q_{j-1},k_j}\backslash \{q_{j}\}$, then there is $1\leq k_{j+1}\leq n_{q_j}$ such that $C_{q_j,k_{j+1}}\cap M$ is infinite, and let $q_{j+1}\in C_{q_j, k_{j+1}} \cap M$ such that $[q_j, q_{j+1}]_M=\{q_j, q_{j+1}\}$. Therefore, $M$ contains an aligned sequence of points, and so we can apply Lemma \ref{Lemma:Aligned}. \qedhere
\end{enumerate}
\end{proof}

\section{Embedding of \texorpdfstring{$c_0$}{c0} in pointwise norm-attaining Lipschitz functions}\label{Section-PNA-c0}

Let us start by recalling the facts \ref{S}--\ref{SSS}, Theorem \ref{theorem:negative-example-ell1}, and Theorem \ref{thm:counterexamples-c0-SNA-2}. 

\begin{itemize}
\itemsep0.3em
\item For any infinite pointed metric space $M$, we have that $c_0$ is always isomorphically contained in $\sna(M)$. Moreover, if $M$ is not uniformly discrete, then $c_0$ is isometrically contained in $\sna(M)$.  
\item However, there exist some metric spaces $M$ such that $c_0$ is not isometrically contained in $\sna(M)$.
\item There are metric spaces $M$ such that $\ell_\infty$ is not isometrically contained in $\pna(M)$.
\end{itemize}

These motivate us to pose the following natural question.

\begin{question}\label{Q:Main}
Is it true for every infinite pointed metric space $M$ that $c_0$ is isometrically contained in $\pna(M)$?
\end{question}

In this section, we provide four distinct results, namely Proposition \ref{prop:pna_c_0_33}, Theorems \ref{theorem:bounded-counterexample2}, \ref{thm:counter_to_bounded-counterexample2}, and \ref{theorem:unbounded-counterexample}, all of which concern the isometric embedding of $c_0$ into $\pna(M)$ in the setting of a uniformly discrete metric space $M$. Consequently, we observe that for all the explicitly known metric spaces $M$ such that $c_0$ is not isometrically embedded in $\sna(M)$, it is actually possible to embed $c_0$ isometrically in $\pna(M)$. Finally, we prove the main result, Theorem \ref{Main-Theorem-c0-PNA}, which gives a partial answer to Question \ref{Q:Main} by showing that for any infinite pointed metric space $M$ there is a metric space $M_0 \subseteq M$ for which $\sna(M_0)$ contains an isometric copy of $c_0$. 

We start with the following analogue of Proposition \ref{prop:c0_characterizationQ} for pointwise norm-attaining Lipschitz functions. The proof of the result is an easy adaptation of Proposition \ref{prop:c0_characterizationQ} and \cite[Lemma 3.3]{DMQR23}. 

\begin{proposition}\label{prop:pna_c_0_33}
Let $M$ be an infinite uniformly discrete pointed metric space, and let $\{p_n\}_{n=1}^{\infty}$ be a sequence of distinct points in $M$. The following claims are equivalent:
\begin{itemize}
\item[\textup{(a)}] $\pna(M)$ contains $c_0$ isometrically with a sequence of functions $\{f_n\}_{n=1}^{\infty}$ that are isometrically a canonical basis of $c_0$ inside $\pna(M)$ and given by
\[
f_n( p_{n} ) = a_n, \quad f_n (p) = 0 \, \text{ if } p \in M \setminus \{p_n\}
\]
for some numbers $\{a_n\}_{n=1}^{\infty}\subset \bbr$.
\item[\textup{(b)}] For each $m\neq n\in \bbn$, the following holds
$$d(p_n, p_m)\geq R(p_n) + R(p_m).$$
\end{itemize}
\end{proposition}

Again, there are simple uniformly discrete spaces that do not satisfy the condition in Proposition \ref{prop:pna_c_0_33}, such as the space $M$ where $d(x,y)=1$ for all $x\neq y \in M$. However, this metric space $M$ satisfies the assumption in Theorem \ref{theorem:c0-condition-1-improved}; hence $\sna (M)$ contains an isometric copy of $c_0$.

The following result shows that there is a wide class of metric spaces $M$, including all the spaces from Corollary \ref{Thm-Negative-c0-SNA}, such that $c_0$ is isometrically contained in $\pna(M)$. In particular, this is the case for the metric space considered in \cite[Theorem 4.1]{DMQR23} ($M=\{p_n\}_{n=1}^{\infty}$ endowed with the metric $d(p_n, p_m):=1+1/ {\max\{m,n\}}$), which is not covered by Proposition \ref{prop:pna_c_0_33}.



\begin{theorem}\label{theorem:bounded-counterexample2}
Let $M$ be an infinite uniformly discrete pointed metric space containing a sequence of distinct points $\{p_n\}_{n=1}^{\infty}$ satisfying these conditions:
\begin{enumerate}
\itemsep0.3em
\item[\textup{(i)}] $d(p_n, p_m)$ is monotonically decreasing with respect to $m$ and $n$; 
\item[\textup{(ii)}] $\displaystyle \displaystyle R(p_k) \geq \lim_m d(p_k,p_m) - \frac{\lim_n\lim_m d(p_n,p_m)}{2}$ for each $k \in \mathbb{N}$. 
\item[\textup{(iii)}] $\displaystyle \lim_m d(p_k, p_m) + \lim_m d(p_\ell, p_m) \leq \lim_n \lim_m d(p_n,p_m) + d(p_k, p_\ell)$ for every $k \neq \ell \in \mathbb{N}$. 
\end{enumerate}
Then $\pna(M)$ contains $c_0$ isometrically.
\end{theorem}

\begin{proof}
For simplicity, put $L(n,m) := d(p_n, p_m)$ for each $n$ and $m$, $L(n) := \lim_m L(n,m)$, and $L = \lim_n L(n)$. Note that $L \leq L(n) \leq L(n,m)$ for every $n,m \in \mathbb{N}$. 
Let $\{r_n\}_{n=1}^{\infty}$ be the sequence of all prime numbers in order. 
Let us point out the following inequalities:
\begin{equation}\label{eq:ineqLR1}
\inf_{j \in \N}  R(p_{j}) \leq R(p_{n}) \leq  L(n) \leq L(n, m).
\end{equation} 
From the assumption (ii), we have 
\begin{equation}\label{eq:ineqLR2} 
L(n) - R(p_{n}) \leq L(n) - \Big( L(n) - \frac{L}{2} \Big) = \frac{L}{2} \, \text{ for every } n \in \mathbb{N}. 
\end{equation} 
As $L = \inf_n L(n)$, notice from \eqref{eq:ineqLR2} that 
\begin{equation}\label{eq:ineqLR3}
L \leq 2 R (p_{n}) \, \text{ for every } n \in \mathbb{N}. 
\end{equation} 
Finally, the assumption (iii) is equivalent to say 
\begin{equation}\label{eq:ineqLR4}
L(n) + L(m) \leq L + L(n,m) \, \text{ for every } n\neq m \in \mathbb{N}. 
\end{equation} 

For each $n\in\bbn$, define $f_n$ as follows:
$$f_n(x):=\begin{cases}
L(r_n) - L/2  ,\quad &\text{if } x= p_{r_n}\\ 
-L/2 ,\quad &\text{if } x= p_{r_n^m} \text{ for some $m>1$}\\
0 ,\quad &\text{otherwise}. 
\end{cases}$$

Let $(a_n)_{n=1}^\infty \in c_0$. We will show that $f:= \sum_{n=1}^\infty a_nf_n$ belongs to $\PNA(M)$ and $\|f\| = \|(a_n)_n\|$. Fix any $n_0 \in \N$ such that $\|(a_n)_n\| = |a_{n_0}|$. Again, we split the cases:

(1) $p = p_{r_n}$, $q = p_{r_\ell^m}$ for some $n, \ell, m \in \mathbb{N}$. If $m =1$ and $n \neq \ell$, then by \eqref{eq:ineqLR2} and \eqref{eq:ineqLR4},  
\begin{align*} 
|S(f,p,q)| &= \frac{| a_n ( L(r_n) - \frac{L}{2} ) - a_\ell ( L(r_\ell) - \frac{L}{2}  ) |} {L(r_n, r_\ell)} \leq  \| a\| \frac{ (L(r_n) + L(r_\ell) - L)}{ L(r_n,r_\ell)} \leq \|a\| 
\end{align*} 

If $m>1$, then 
\begin{align*}
|S(f,p,q)| = \frac{| a_n ( L(r_n) - \frac{L}{2} ) - (a_\ell (-\frac{L}{2} )  ) |}{L(r_n, r_\ell^m)} \leq \|a\| \frac{ (L(r_n) - \frac{L}{2}) + \frac{L}{2}  } {L(r_n, r_\ell^m)} \leq \| a\| 
\end{align*} 
by \eqref{eq:ineqLR1}. 

(2) $p = p_{r_n^k}$, $q = p_{r_\ell^m}$ for some $k, m >1$ and $n, \ell \in \mathbb{N}$. 
In this case, we have
$$
|S(f,p,q)| = \frac{| - a_n \frac{L}{2} - (-a_\ell \frac{L}{2}) |}{L(r_n^k, r_\ell^m)} \leq \|a\| \frac{L}{L(r_n^k, r_\ell^m)} \leq \|a\|
$$
by \eqref{eq:ineqLR1} again. 

(3) $p,q \notin \{p_{r_{n}^m}:\, n,m\in\bbn\}$. \\
Clearly, we get $\sum_{n=1}^\infty a_nf_n(p)-\sum_{n=1}^\infty a_nf_n(q) = 0$.

Finally, observe that for $p = p_{r_{n_0}}$ and $q=p_{r_{n_0}^m}$, 
$$
\lim_m |S(f,p_{r_{n_0}}, p_{r_{n_0}^m})| = 
\lim_m \frac{ |a_{n_0} (L(r_{n_0}) -\frac{L}{2}) + a_{n_0} \frac{L}{2} |}{L(r_{n_0}, r_{n_0}^m)} = \frac{|a_{n_0}| L(r_{n_0}) } {L(r_{n_0})} = |a_{n_0}|.
$$
so we can conclude that $\{f_n\}_{n=1}^\infty$ is isometric to the basis of $c_0$ in $\PNA(M)$.
\end{proof}

As mentioned, Theorem \ref{theorem:bounded-counterexample2} can be applied to all the metric spaces $M= \{p_n\}_{n=1}^\infty$ satisfying the conditions in Corollary \ref{Thm-Negative-c0-SNA}.

\begin{cor}\label{cor:counterexample4.1}
There exists a bounded metric space $M$ such that $\PNA(M)$ isometrically embeds $c_0$, while $\sna(M)$ does not.
\end{cor}

\begin{example}\label{ex:counter_to_bounded-counterexample2}
However, there are metric spaces that are not covered by Proposition \ref{prop:pna_c_0_33} and Theorem \ref{theorem:bounded-counterexample2}. For instance, consider the metric space $M = \{0\} \cup \{ p_n\}_n$ endowed with the metric given by 
$d(p_n,p_m) = 1 + 1/(n+m)$ and $d(p_n,0) = 1/2 + 1/n$ for every $n, m \in \mathbb{N}$. It is easy to see that there is no sequence of points in $M$ satisfying (b) in Proposition \ref{prop:pna_c_0_33}. Moreover, let $\{q_n\}_n \subseteq M$ such that $d(q_n, q_m)$ is decreasing with respect to $n$ and $m \in \N$. Then $q_n \neq 0$ for every $n \in \N$ and $\{q_n\}_n$ is a subsequence of $\{p_n\}_n$, say $q_k = p_{n_k}$. Then for $k \neq \ell \in \N$, 
\[
1 + \frac{1}{n_k} + 1 + \frac{1}{n_\ell} > 1 + 1 + \frac{1}{n_k + n_\ell};
\]
hence the assumption (3) in Theorem \ref{theorem:bounded-counterexample2} cannot be satisfied either. 
\end{example}

As we will see, however, $\pna (M)$ still contains $c_0$ isometrically for the metric space $M$ from Example \ref{ex:counter_to_bounded-counterexample2}. Note that this metric space is not covered by any of the previous positive results from the article.

\begin{theorem}\label{thm:counter_to_bounded-counterexample2}
Let $M$ be an infinite uniformly discrete pointed metric space containing a sequence of distinct points $\{p_n\}_{n=1}^{\infty}$ satisfying these conditions:
\begin{enumerate}
\itemsep0.3em
\item[\textup{(i)}] For all $n\in\bbn$, $d(p_n, 0)=R(p_n)$.
\item[\textup{(ii)}] The sequence $\{d(p_n, 0)\}_{n=1}^{\infty}$ is decreasing with $n$ and converges to some positive number $D>0$.
\item[\textup{(iii)}] $2D\leq d(p_n, p_m)$ for all $n,m\in\bbn$.
\end{enumerate}
Then $\pna(M)$ contains $c_0$ isometrically.
\end{theorem}

\begin{proof}
Let $\{r_n\}_{n=1}^{\infty}$ be the sequence of all prime numbers in order, and for each $n\in\bbn$, define $f_n \in \Lip (M)$ as follows:
\[
f_n( p_{r_n^m} ) = D \, \text{ for every } m \in \mathbb{N}, \, \text{ and } \, f_n (p) = 0 \, \text{ otherwise}.
\]
Let $a=(a_n)_n \in c_0$, $f_a:=\sum_{n=1}^{\infty} a_n f_n$ and $(p,q)\in\widetilde{M}$ be given. 
Observe that if $p, q \notin \{ p_{r_n^m} : n, m \in \mathbb{N}\}$, then $S(f_a, p,q)=0$. If there exists $n \in \mathbb{N}$ such that $p, q \in \{p_{r_n^m}: m \in \mathbb{N}\}$, then $S(f_a, p,q)=0$. Thus, we consider the following cases: 
\begin{enumerate}
\itemsep0.3em
\item If $p=p_{r_n^m}$ for some $n, m \in \mathbb{N}$ but $q \notin \bigcup_{n=1}^{\infty} R_n$, then $|S(f_a, p, q)|\leq \|a\|$. Moreover, notice that $\sup_{m\in\bbn} |S(f_a, p_{r_n^m}, 0)|=|a_n|$; hence 
$$\sup_{n\in\bbn} \sup_{m\in\bbn} |S(f_a, p_{r_n^m}, 0)|=\|a\|.$$
\item Finally, if there are $(n,\ell) \in \widetilde{\mathbb{N}}$ such that $p = p_{r_n^m}$ and $q = p_{r_\ell^k}$ for some $m, k \in \mathbb{N}$, then
$$|S(f_a, p, q)|\leq \frac{|a_n| D + |a_m| D}{d(p_{r_n^m}, p_{r_\ell^k})}\leq \|a\|.$$
\end{enumerate}
Therefore, $f_a$ attains its pointwise norm at the point $0$ with $\|f_a\| = \|a\|$. This concludes the proof.
\end{proof}

The following result shows that there is a large class of metric spaces $M$ in which $c_0$ is isometrically contained in $\pna(M)$, including the metric spaces from Theorem \ref{thm:counterexamples-c0-SNA-2} and \cite[Remark 10]{CJ17}, and the metric space considered in \cite[Theorem 4.4]{DMQR23}: $M = \{p_n\}_{n=1}^\infty$ is a metric space endowed with the metric $d$ defined by
\begin{displaymath}
d(p_n,p_m)=\left\{\begin{array}{@{}cl}
\displaystyle \phantom{.} n+m - \eps_{\max\{n,m\}} & \text{if } n \neq m \\
\displaystyle \phantom{.} n & \text{if } m = 0 \\
\displaystyle \phantom{.} m & \text{if } n = 0 \\
\displaystyle \phantom{.} 0 & \text{otherwise},
\end{array} \right.
\end{displaymath}
where $\{\eps_n\}_{n=1}^\infty\subset\, ]0,\, \frac{1}{2}[$ is a strictly increasing sequence of positive numbers.


\begin{theorem}\label{theorem:unbounded-counterexample}
Let $M$ be an infinite pointed metric space such that there exists a sequence $\{(p_n,\eps_n)\}_{n=1}^\infty \subseteq M \times [0,+\infty[$ satisfying the following assertions:
\begin{enumerate}
\itemsep0.3em
\item[\textup{(i)}] $\dfrac{d(p_n,0) - \eps_n + d(p_m,0) - \eps_m}{d(p_n,p_m)} \leq 1$ for every $n,m \in \N$, and moreover, for any fixed $n \in \N$, the left-hand side converges to $1$ as $m$ tends to $\infty$,
\item[\textup{(ii)}] $0 \leq d(p_n,0)-\eps_n \leq R(p_n)$ for every $n\in\bbn$. 
\end{enumerate}
Then $\pna(M)$ contains $c_0$ isometrically.
\end{theorem}

\begin{proof}
Let $\{r_n\}_{n=1}^\infty \subseteq \N$ be a sequence of increasing prime numbers. Define $f_n \in \lip(M)$ for each $n \in \N$ by
\begin{displaymath}
f_n(x)=\left\{\begin{array}{@{}cl}
\displaystyle \phantom{.} \phantom{\int} d(p_{r_n},0) - \eps_{r_n} \phantom{\int} & \text{if } x=p_{r_n} \\
\displaystyle \phantom{.} \phantom{\int} -(d(p_{r_n^m},0) - \eps_{r_n^m}) \phantom{\int} & \text{if } x=p_{r_n^m} \text{ for some } m \geq 2 \\
\displaystyle \phantom{.} \phantom{\int}0\phantom{\int} & \text{otherwise}.
\end{array} \right.
\end{displaymath}

Let $(a_n)_{n=1}^\infty \in c_0$ be given. We will show that $f:=\sum_{n=1}^\infty a_nf_n \in \PNA(M)$ and $\|f\| = \|(a_n)_{n=1}^\infty\|$. Fix any $n_0 \in \N$ such that $\|(a_n)_n\| = |a_{n_0}|$. Let $p_j,p_k \in M$ with $j \neq k$ be given. Then, either we have \textup{(i)} $j = r_{n_1}^{m_1}$ and $k = r_{n_2}^{m_2}$ for some $(n_1,m_1) \neq (n_2,m_2)$, \textup{(ii)} $j = r_{n_1}^{m_1}$ for some $n_1, m_1 \in \N$ and $k \neq r_{n_2}^{m_2}$ for any $n_2, m_2 \in \N$, \textup{(iii)} $j,k \neq r_{n_1}^{m_1}$ for any $n_1, m_1 \in \N$. For the case \textup{(i)}, we have
$$
|S(f, p_{r_{n_1}^{m_1}}, p_{r_{n_2}^{m_2}})|  \leq \frac{|a_{n_1}|(d(p_{r_{n_1}^{m_1}},0) - \eps_{r_{n_1}^{m_1}}) + |a_{n_2}| (d(p_{r_{n_2}^{m_2}},0) - \eps_{r_{n_2}^{m_2}})}{d(p_{r_{n_1}^{m_1}}, p_{r_{n_2}^{m_2}})} \leq |a_{n_0}|,
$$
and the rest cases can be done similarly. Moreover, observe that
$$
\lim_m |S(f, p_{r_{n_0}},p_{r_{n_0}^m})|  = \lim_m |a_{n_0}|\, \frac{d(p_{r_{n_0}},0) - \eps_{r_{n_0}}+d(p_{r_{n_0}^m},0) - \eps_{r_{n_0}^m}}{d(p_{r_{n_0}},p_{r_{n_0}^m})} = |a_{n_0}|,
$$
and this shows that $\{f_n\}_{n=1}^\infty \subseteq \PNA(M)$ is isometric to the basis of $c_0$.
\end{proof}

\begin{remark}
If $M$ and $\{p_n\}_{n=1}^{\infty}$ are as in Theorem \ref{theorem:bounded-counterexample2}, then, if we set $p_1=0$ and for each $n\in\bbn$ we define
$$\varepsilon_n:= d(p_1, p_n) + \frac{\lim_n \lim_n d(p_n, p_m)}{2} - \lim_m d(p_n, p_m),$$
then we get that the sequence $\{(p_n, \eps_n)\}_{n=1}^\infty$ satisfies the conditions in Theorem \ref{theorem:unbounded-counterexample}. However, the basis functions of $c_0$ constructed by Theorem \ref{theorem:bounded-counterexample2} are different from the ones given by Theorem \ref{theorem:unbounded-counterexample}.  
\end{remark} 

As we mentioned above, Theorem \ref{theorem:unbounded-counterexample} can be applied to metric space considered in \cite[Theorem 4.4]{DMQR23}; hence the following corollary is immediate.

\begin{cor}
There exists a proper metric space $M$ for which $\PNA(M)$ contains an isometric copy of $c_0$, while $\sna(M)$ does not.
\end{cor}

\begin{example}\label{example:motivation-main-result}
Consider the metric space $M = \{p_n\}_{n=1}^\infty$ whose metric $d$ is given by $d(p_n,p_m) = 2 - \frac{1}{3^n}-\frac{2}{3^m}$ for every $n<m$. Then this metric space $M$ does not satisfy any of our previous results. Indeed, it is clear that it does not satisfy Propositions \ref{prop:c0_characterizationQ} and \ref{prop:pna_c_0_33} and Theorems \ref{theorem:bounded-counterexample2} and \ref{thm:counter_to_bounded-counterexample2}. Moreover, if it satisfied Theorem \ref{theorem:unbounded-counterexample}, then in order to ensure convergence to $1$, up to subsequence, for each $n\in\bbn$ we get $\varepsilon_n= \frac{2}{3}-\frac{1}{3^n}$ and also $\lim_n \varepsilon_n = \frac{2}{3}$, but then the fractions would be greater than $1$, which is a contradiction. Also, it cannot satisfy Theorem \ref{theorem:c0-condition-1-improved}, as if it did, there would exist strictly increasing sequences of naturals $\{\sigma_n\}_n$, $\{\tau_n\}_n$ such that for all $n\in \bbn$, $\sigma_n<\tau_n<\sigma_{n+1}$, and for all $n<m$, it must hold that
$$2-\frac{1}{3^{\sigma_n}} - \frac{2}{3^{\tau_n}} + 2-\frac{1}{3^{\sigma_m}} - \frac{2}{3^{\tau_m}} \leq 2\left(2-\frac{1}{3^{\sigma_n}} - \frac{2}{3^{\sigma_m}}\right).$$
But if we fix $n\in\bbn$ and tend $m\to\infty$, that inequality would imply that
$$\frac{1}{3^{\sigma_n}}\leq \frac{2}{3^{\tau_n}},$$
which is a contradiction. Finally, arguing similarly, the condition (4) from Theorem \ref{thm:1+1/n+1/m} also leads to a contradiction. However, we will see that $c_0$ is isometrically contained in $\sna(M)$, as shown in the Subcase I-(ii) of the proof of the main theorem.
\end{example}

Finally, we finish this section with our main result, which gives a partial positive answer to Question \ref{Q:Main}. 
This highlights once again how pointwise norm-attaining Lipschitz functions differ from strongly norm-attaining functions (see Theorem \ref{thm:counterexamples-c0-SNA-2}).

\begin{theorem}\label{Main-Theorem-c0-PNA}
Let $M$ be an infinite pointed metric space. Then there exists a metric subspace $M_0 \subseteq M$ such that $c_0$ is isometrically contained in $\pna(M_0)$. 
\end{theorem} 

\begin{proof}
If $M$ is not uniformly discrete, $\pna(M)$ contains an isometric copy of $c_0$ by \ref{SS}. Suppose that $M$ is uniformly discrete, and let us divide the proof into two cases: 

\textit{Case I}: $M$ contains a bounded sequence $\{p_n\}_n$. Then by compactness, passing to a subsequence if necessary, we may assume that 
\[
\lim_{n,m} d(p_n, p_m) = \lim_{n}\lim_{m} d(p_n, p_m) = L.
\]
For simplicity, put 
\[
\phi(n,m) := d(p_n,p_m) - L, \, \text{ and } \, \psi(n) := \lim_{m} \phi(n,m). 
\] 
Passing to a further subsequence, we assume that $|\phi(n,m)| < L/5$ for every $n \neq m\in \mathbb{N}$. 
Consider 
\begin{align*}
&A = \{ (n,m) \in \mathbb{N}^{[2]} : \phi(n,m) \geq \psi(n) + \psi(m) \}, \\
&B = \{ (n,m) \in \mathbb{N}^{[2]} : \phi(n,m) < \psi(n) + \psi(m) \}. 
\end{align*} 
Using Ramsey's theorem, there exists $C \in \{A,B\}$ and an infinite subset $S \subseteq \mathbb{N}$ such that $S^{[2]} \subseteq C$. 
We distinguish two subcases.  

\textit{Subcase I-\textup{(i)}}: $C=A$. That is, $\phi(n,m) \geq \psi(n) + \psi(m)$ for every $n \neq m \in S$. Define $M_0 := \{ p_n \}_{n \in S}$, put $p_1 = 0$ and set $\eps_n = L/2 + \phi(1,n) - \psi(n)$. Then 
\[
\frac{ d(0,p_n)-\eps_n + d(0,p_m) -\eps_m }{d(p_n,p_m)} = \frac{L +\psi(n) + \psi(m)}{L+\phi(n,m)} \leq 1
\]
and the left-hand side tends to $1$ when $m$ goes to infinity. Moreover, since $d(p_n,p_m) - L \geq \psi(n) + \psi(m)$, we have 
\[
R(p_n) \geq L + \psi(n) - \frac{L}{5} = \frac{4L}{5} + \psi(n),
\] 
so $0 \leq d(0, p_n) - \eps_n = L/2 + \psi(n) < R(p_n)$. Now, applying Theorem \ref{theorem:unbounded-counterexample} to $M_0 = \{ p_n \}_{n \in S}$, we conclude that $c_0$ is isometrically contained in $\pna (M_0)$. 

\textit{Subcase I-\textup{(ii)}}: $C=B$. That is, $\phi(n,m) < \psi(n)+\psi(m)$ for every $n \neq m \in S$. Put $M_0 := \{ p_n \}_{n \in S}$. For notational simplicity, we can assume up to relabeling that $S=\bbn$. 
Let us construct sequences $(\sigma_n)_n$ and $(\tau_n)_n$ of natural numbers inductively: let $\sigma_1 = 1$ and $\tau_1 = 2$. Next, if $\sigma_n$ and $\tau_n$ are given for some $n \in \mathbb{N}$, define 
\[
\eps_n = \frac{ - \phi(\sigma_n, \tau_n) + \psi(\sigma_n) + \psi(\tau_n) }{6},
\]
and find $\sigma_{n+1}$ and $\tau_{n+1} : = \sigma_{n+1} +1$ so that the following inequalities hold: if $i \geq \sigma_{n+1}$ and $j \geq \tau_{n+1} = \sigma_{n+1} + 1$, then 
\begin{align*}
|\phi( i, j )| < \eps_n, \, |\psi(i)| < \eps_n, \, |\phi(\sigma_n, i) - \psi(\sigma_n)| < \eps_n, \, \text{ and }\, |\phi(\tau_n, i) - \psi(\tau_n)| < \eps_n. 
\end{align*} 
Consider for each $n \in \N$
$$
f_n (p) 
:=\begin{cases}
\frac{1}{2} (L + \phi(\sigma_n, \tau_n) +\psi(\sigma_n) - \psi(\tau_n)), \quad &\text{if } p= p_{\sigma_n}\\
-\frac{1}{2} (L + \phi(\sigma_n, \tau_n) - \psi(\sigma_n) + \psi(\tau_n)), \quad &\text{if } p= p_{\tau_n}\\
0, &\text{else.}
\end{cases}
$$

Let $(a_n)_{n} \in c_0$ with $|a_{k_0}| = \|a\|$ for some $k_0 \in \mathbb{N}$, and let $h = \sum_{n=1}^\infty a_n f_n$. Let $u \neq v \in M_0$ be given. We will compute the possible values of the slope $|S(h, u, v)|$:
\begin{itemize}
\itemsep0.3em
\item If $u,v\notin \bigcup_{n\in\bbn} \{p_{\sigma_n}, p_{\tau_n}\}$, then $S(h, u, v)=0$.
\item If $u\in \{p_{\sigma_n}, p_{\tau_n}\}$ for some $n\in\bbn$ and $v\notin \bigcup_{n\in\bbn} \{p_{\sigma_n}, p_{\tau_n}\}$, then 
$$|S(h, u, v)|\leq |a_n| \frac{|f_n(u)|}{R(u)} \leq |a_n| \frac{|f_n(u)|}{4L/5}\leq\|a\|.$$
\item If there is some $n\in\bbn$ such that $u=p_{\sigma_n}$ and $v=p_{\tau_n}$, then 
$$|S(h, u, v)|=|a_n|\leq \|a\|.$$
\item Finally, if there exist $n<m\in \bbn$ such that $u\in\{p_{\sigma_n}, p_{\tau_n}\}$ and $v\in\{p_{\sigma_m}, p_{\tau_m}\}$, we have 4 possibilities:
\end{itemize}
\begin{enumerate}
\item If $u=p_{\sigma_n}$ and $v=p_{\sigma_m}$, then
$$|S(h, u, v)|\leq \|a\|\frac{2L+\phi(\sigma_n, \tau_n) + \phi(\sigma_m, \tau_m) + \psi(\sigma_n) + \psi(\sigma_m) - \psi(\tau_n) - \psi(\tau_m)}{2L+2\phi(\sigma_n, \sigma_m)}.$$
\item If $u=p_{\sigma_n}$ and $v=p_{\tau_m}$, then
$$|S(h, u, v)|\leq \|a\|\frac{2L+\phi(\sigma_n, \tau_n) + \phi(\sigma_m, \tau_m) + \psi(\sigma_n) - \psi(\sigma_m) - \psi(\tau_n) + \psi(\tau_m)}{2L+2\phi(\sigma_n, \tau_m)}.$$
\item If $u=p_{\tau_n}$ and $v=p_{\sigma_m}$, then
$$|S(h, u, v)|\leq \|a\|\frac{2L+\phi(\sigma_n, \tau_n) + \phi(\sigma_m, \tau_m) - \psi(\sigma_n) + \psi(\sigma_m) + \psi(\tau_n) - \psi(\tau_m)}{2L+2\phi(\tau_n, \sigma_m)}.$$
\item Finally, if $u=p_{\tau_n}$ and $v=p_{\tau_m}$, then
$$|S(h, u, v)|\leq \|a\|\frac{2L+\phi(\sigma_n, \tau_n) + \phi(\sigma_m, \tau_m) - \psi(\sigma_n) - \psi(\sigma_m) + \psi(\tau_n) + \psi(\tau_m)}{2L+2\phi(\tau_n, \tau_m)}.$$
\end{enumerate}
We claim that each fraction in the right-hand side in (1)--(4) is bounded above by $\|a\|$. Since the proofs are similar, we only give the proof for (1). Indeed, observe that 
\begin{align*}
-2\phi(\sigma_n&, \sigma_m) + \phi(\sigma_n, \tau_n) + \phi(\sigma_m, \tau_m) + 2\psi(\sigma_n) + \psi(\sigma_m) - \psi(\tau_m)\\
&= \phi(\sigma_n, \tau_n) + \phi(\sigma_m, \tau_m) + \psi(\sigma_m) + (-\psi(\tau_m)) + 2(\psi(\sigma_n) - \phi(\sigma_n, \sigma_m))\\
&\leq \phi(\sigma_n, \tau_n) +5\varepsilon_n < \psi(\sigma_n) + \psi(\tau_n).
\end{align*}
This implies that $|S(h,u,v)|\leq \|a\|$ in the case of (1). Finally, note that $|S(h, p_{\sigma_{k_0}}, q_{\sigma_{k_0}})|=\|a\|$, and so $h$ is strongly norm-attaining and has norm $\|a\|$.

\textit{Case II}: $M$ does not contain a bounded sequence. Motivated by \cite[Theorem 5]{CJ17}, we will inductively construct an appropriate sequence of points $\{p_n\}_{n=1}^{\infty}\subset M$ and a sequence of positive numbers $\{c_n\}_{n=1}^{\infty}\subset (0,\infty)$ as follows. Fix some $p_1\in M$ and some $c_1>0$. Now, if for some $n\in\bbn$, the points $\{p_k\}_{k=1}^n\subset M$ and the numbers $\{c_k\}_{k=1}^{n}\subset (0,\infty)$ are given, find $p_{n+1}\in M$ satisfying that for all $1\leq k\leq n$,
$$d(p_k, p_{n+1})>n \max\{d(p_a, p_b) + c_a:\, 1\leq a,b\leq n\},$$
and define
$$c_{n+1}:= \max\{d(p_a, p_b):\, 1\leq a,b\leq n+1\} - \max\{d(p_a, p_b) + c_a:\, 1\leq a,b\leq n\}.$$
From now on, we will restrict ourselves to the metric subspace $M_0=\{p_n\}_{n=1}^{\infty}$. Note that, for all $n\in\bbn$, there exist $1\leq k_n, j_n<n$ such that the following holds:
$$c_n>0,\quad R(p_n)=d(p_{k_n}, p_n),\quad \text{and}\quad \max\{d(p_a, p_b):\, 1\leq a,b\leq n\} = d(p_{j_n}, p_n).$$

Given $n<m\in \bbn$, 
\begin{align*}
c_n+c_m &= c_n  + \max\{d(p_a, p_b):\, 1\leq a,b\leq m\} - \max\{d(p_a, p_b) + c_a:\, 1\leq a,b\leq m-1\} \\
&\leq c_n + d(p_{j_m}, p_m) - d(p_{j_m}, p_n) - c_n\leq d(p_n, p_m).
\end{align*}

As a consequence, note also that for all $n\in\bbn$, 
$$c_n\leq c_n + c_{k_n}\leq d(p_n, p_{k_n}) = R(p_n).$$

Moreover, note that for $n<m\in \bbn$, 
\begin{align*}
1&\geq \frac{c_n+c_m}{d(p_n, p_m)} \\
&= \frac{c_n + \max\{d(p_a, p_b):\, 1\leq a,b\leq m\} - \max\{d(p_a, p_b) + c_a:\, 1\leq a,b\leq m-1\}}{d(p_n, p_m)}\\
&= \frac{d(p_{j_m}, p_m)}{d(p_n, p_m)} - \frac{\max\{d(p_a, p_b) + c_a:\, 1\leq a,b\leq m-1\} - c_n}{d(p_n, p_m)}\\
&\geq 1 - \frac{\max\{d(p_a, p_b) + c_a:\, 1\leq a,b\leq m-1\}}{d(p_n, p_m)} \geq 1-\frac{1}{m-1},
\end{align*}
which converges to $1$ when $m\to\infty$.

Denote now $\{r_n\}_{n=1}^{\infty}$ to the increasing sequence of all prime numbers, and for each $n\in\bbn$, define the Lipschitz function $f_n: M_0 \rightarrow \bbr$ by
$$f_n(x):=\begin{cases}
c_{r_n},\quad &\text{if }x=p_{r_n},\\
-c_{r_n^m},\quad &\text{if }x=p_{r_n^m}, \text{ for some }m>1,\\
0,\quad &\text{else.}
\end{cases}$$

Let $(a_n)_n \in c_0$ be such that $\|a\|=|a_{k_0}|$ for some $k_0\in\bbn$, and let $f:=\sum_{n=1}^{\infty} a_n f_n$. We will compute the slopes $|S(f, x, y)|$ for $(x, y)\in \widetilde{M_0}$ studying all the possible cases:
\begin{enumerate}
\item If $x,y\notin \{p_{r_j^k}:\, j,k\in\bbn\}$, then $|S(f, x, y)|=0$.
\item If $x=p_{r_n^m}$ for some $n,m\in \bbn$ but $y\notin \{p_{r_j^k}:\, j,k\in\bbn\}$, then
$$|S(f, x, y)|\leq |a_n| \frac{c_{r_n^m}}{R(p_{r_n^m})}\leq \|a\|.$$
\item If there are $n,m,k\in\bbn$ with $m\neq k$ such that $x=p_{r_n^m}$ and $y=p_{r_n^k}$, then
$$|S(f, x, y)|\leq |a_n|\frac{c_{r_n^m} + c_{r_n^k}}{d(p_{r_n^m}, p_{r_n^k})}\leq \|a\|.$$
Moreover, note that 
$$\lim_{m\to\infty} |S(f, p_{r_{k_0}}, p_{r_{k_0}^m})| = \lim_{m\to\infty} |a_{k_0}|\frac{c_{r_{k_0}} + c_{r_{k_0}^m}}{d(p_{r_{k_0}}, p_{r_{k_0}^m})}=\|a\|.$$
\item Finally, if there exist $n,m,j,k\in\bbn$ with $n\neq j$ such that $x=p_{r_n^m}$ and $y=p_{r_j^k}$, then
$$|S(f, x, y)|\leq \|a\| \frac{c_{r_n^m} + c_{r_j^k}}{d(p_{r_n^m}, p_{r_j^k})}\leq \|a\|.$$
\end{enumerate}
Therefore, $f$ has norm $\|a\|$, and it attains its pointwise norm at the point $p_{r_{k_0}}$.  
\end{proof}


\section{Embedding of \texorpdfstring{$\ell_1$}{l1} into norm-attaining Lipschitz functions}\label{Section-PNA-Ell-1}

In this section, we will provide some further spaceability results on the set of strongly norm-attaining and pointwise norm-attaining Lipschitz functions. These results will, once more, show how different these two sets are. First, we will provide some sufficient conditions and necessary conditions in order to embed $\ell_1$ isometrically into $\sna(M)$. Second, we will show that, in contrast, none of those necessary conditions for $\ell_1$-embedding into the set of strongly norm-attaining Lipschitz functions $\sna(M)$ are not necessary anymore for the set of pointwise norm-attaining Lipschitz functions $\pna(M)$.

For simplicity, given a point $p$ in a metric space $M$ and $r>0$, we denote by $B(p,r)$ the closed ball centered at $p$ with radius $r$. 



\begin{theorem}\label{theorem:ell1-tents-1}
Let $\Gamma$ be an index set of cardinality $2^{\aleph_0}$, and let $M$ be a pointed metric space containing a set $\{p_\gamma, q_\gamma \}_{\gamma\in \Gamma}$ such that $d(p_\alpha,p_\beta)\geq d(p_\alpha, q_\alpha)+d(p_\beta, q_\beta)$ for all $\alpha\neq\beta\in\Gamma$. Then $\sna(M)$ contains $\ell_1$ isometrically.
\end{theorem} 

\begin{proof}
Given $\Gamma$, let $S=\{s^{(\gamma)}\}_{\gamma\in\Gamma}$ be the set of all sequences of signs $s^{(\gamma)}:=(s^{(\gamma)}_1, s^{(\gamma)}_2,\ldots)$ such that $s^{(\gamma)}_j\in\{-1,1\}$ for all $j\in\bbn$. 

For each $\gamma\in\Gamma$, let $f_\gamma:M\rightarrow \mathbb{R}$ be the Lipschitz function defined as
$$
f_\gamma(p):= \max\{0, d(p_\gamma,q_\gamma) - d(p_\gamma,p)\}.
$$
For each $n\in\mathbb{N}$, define the function $e_n:M \rightarrow \mathbb{R}$ as
$$
e_n(p):=\begin{cases}
s^{(\gamma)}_n f_\gamma(p),\quad &\text{if $p\in B(p_\gamma, d(p_\gamma, q_\gamma))$ for some $\gamma\in\Gamma$}\\
0,\quad &\text{else.}
\end{cases}
$$
It is not difficult to check $\|e_n\| = 1$. 
Let $a=(a_n)_n \in \ell_1$ and $f:=\sum_{n=1}^{\infty}a_n e_n$. 
Then $\|f\|\leq \|a\|$. We will show that $f\in\sna(M)$ with $\|f\|=\|a\|$. Indeed, let $(b_n)_n$ be the sequence defined as follows: for each $n\in\bbn$, if $a_n=0$, then $b_n=1$, and else, $b_n=\sign(a_n)$.
and let $\gamma_0\in\Gamma$ be such that $s^{(\gamma_0)}=b$. Then, $f(q_{\gamma_0})=0$ and $f(p_{\gamma_0})=\|a\|  d(p_{\gamma_0}, q_{\gamma_0})$.
\end{proof}

It is shown in \cite[Proposition 5.1, Theorem 5.2]{DMQR23} that if a metric space $M$ satisfies that $\dens(M)=\Gamma$ for some infinite cardinal $\Gamma$, then there exists a set $\{p_\gamma, q_\gamma\}_{\gamma \in \Gamma}$ such that $d(p_\alpha,p_\beta) \geq d(p_\alpha,q_\alpha)+d(p_\beta,q_\beta)$ for all $\alpha\neq\beta \in \Gamma$. By Theorem \ref{theorem:ell1-tents-1}, we obtain the following immediate consequence. 

\begin{corollary}\label{cor:density_character}
If $X$ is a normed space of density character $\dens(X)\geq 2^{\aleph_0}$, then $\sna(X)$ contains $\ell_1$ isometrically.
\end{corollary}

\begin{remark}
Combining Corollary \ref{cor:density_character} with the fact \ref{KRS}, we observe that under Continuum Hypothesis, a Banach space $X$ is not separable if and only if $\sna(X)$ contains $\ell_1$.
\end{remark}

However, we will find that the condition in  Theorem \ref{theorem:ell1-tents-1} is sufficient, but not necessary. As a matter of fact, there are metric spaces $M$ not satisfying those conditions but for which $\ell_1$ is also isometrically contained in $\sna(M)$, as the following result shows.

\begin{proposition}\label{theorem:ell1-discrete-1}
Let $\Gamma$ be an index set of cardinality $2^{\aleph_0}$, and let $M=\bigcup_{\gamma\in\Gamma}\{p_\gamma, q_\gamma\}$ such that for all $\alpha\neq\beta\in\Gamma$, $|\{p_\alpha, q_\alpha, p_\beta, q_\beta\}|=4$, and for all $(p,q)\in \widetilde{M}$, $d(p,q)=1$. Then $\sna(M)$ contains $\ell_1$ isometrically. 
\end{proposition} 

\begin{proof}
Given $\Gamma$, let $S=\{s^{(\gamma)}\}_{\gamma\in\Gamma}$ be the set of all sequences of signs $s^{(\gamma)}:=(s^{(\gamma)}_1, s^{(\gamma)}_2,\ldots)$ such that $s^{(\gamma)}_j\in\{-1,1\}$ for all $j\in\bbn$.  
For each $\gamma\in\Gamma$, let $f_\gamma:M\rightarrow \mathbb{R}$ be the Lipschitz function defined as
\[
f_\gamma (p_\gamma) = 1/2, \quad  f_\gamma(q_\gamma) = -1/2, \quad f_\gamma(p) = 0 \text{ if } x \in M \setminus \{p_\gamma, q_\gamma\}. 
\]

For each $n\in\mathbb{N}$, define the Lipschitz function $e_n:M\rightarrow \mathbb{R}$ as
$$e_n(p):=\begin{cases}
s^{(\gamma)}_n f_\gamma(p),\quad &\text{if $p\in B(p_\gamma, d(p_\gamma, q_\gamma))$ for some $\gamma\in\Gamma$}\\
0,\quad &\text{else.}
\end{cases}$$

It is not difficult to check that $\|e_n\|=1$ for every $n \in \mathbb{N}$. 
Let $a=(a_n)_n \in \ell_1$ and   $f:=\sum_{n=1}^{\infty}a_n e_n$. Then $\|f\|\leq \|a\|$. We will now show that $f\in\sna(M)$ with $\|f\|=\|a\|$. Indeed, let $(b_n)_n$ be the sequence defined as follows: for each $n\in\bbn$, if $a_n=0$, then $b_n=1$, and else, $b_n=\sign(a_n)$. Now, take $\gamma_0\in\Gamma$ be such that $s^{(\gamma_0)}=b$. Then, $S(f, p_{\gamma_0}, q_{\gamma_0})= \|a\|$. 
\end{proof}

Next, we discuss some necessary conditions for $M$ in order to satisfy that $\sna(M)$ contains $\ell_1$ isometrically, which complement \cite[Theorem 2]{KR22}. The following useful lemma is geometrically clear:

\begin{lemma}\label{lemma:ell1-1}
Let $M$ be a pointed metric space such that $\ell_1$ is isometrically contained in $\sna(M)$ with a basis $\{g_n\}_n\subseteq \sna(M)$ isometrically equivalent to the canonical basis of $\ell_1$. Let $(a_n)_n \in\ell_1$, $I=\{n\in\mathbb{N}:\, a_n\neq 0\}$ be given, and let $f:=\sum_{n=1}^{\infty} a_n g_n \in \sna(M)$. Then, we have the following:
\begin{enumerate}[label=\normalfont{(\alph*)}]
\itemsep0.3em
\item\label{lemma:ell1-1_a} If $S(f, p, q)=\|f\|$ for some $(p,q)\in\widetilde{M}$, then $g_n$ satisfies that $S(g_n, p, q) = \sign(a_n)$ for each $n \in I$. The converse is also true. 
\item\label{lemma:ell1-1_b} Suppose that $a_i \neq 0$ and $a_j \neq 0$ for some $(i,j) \in \widetilde{\mathbb{N}}$. Let $(b_n)_n \in \ell_1$ be given by $b_j =-a_j$ and $b_n=a_n$ for all $n\in\bbn\backslash\{j\}$. If $f$ strongly attains its norm at some $(p,q) \in \widetilde{M}$, then $g := \sum_{n=1}^\infty b_ng_n$ does not strongly attain its norm at the pair $(p,q)$. 
\end{enumerate}
\end{lemma}

\begin{proof}
(a) If $S(f, p, q)=\|f\|$, then 
\begin{align}\label{eq:chains}
& S(f, p, q)=\sum_{n\in I}|a_n| \geq \left| \sum_{n\in I} a_n S(g_n, p, q) \right| = \left| \sum_{n\in I} a_n \frac{g_n(q)-g_n(p)}{d(p, q)} \right| {=}  \left| S(f, p, q)\right|,
\end{align}
This shows that $|a_n|=|a_n S(g_n, p, q)|$; hence $g_n$ strongly attains its norm at the pair $(p, q)$ for all $n\in I$. Moreover, as we can remove some of the absolute values in \eqref{eq:chains}, each $a_n S(g_n, p, q)$ is positive. This completes the proof and all the steps we have done can be reversed.

(b) We can assume that $S(f, p, q) = \|f\|$, since otherwise we can just swap the points $q$ and $p$. By \ref{lemma:ell1-1_a}, we see $S(g_{i}, p, q) = \sign(a_i)$ and $S(g_{j}, p, q)=\sign(a_j)$. However, if $g$ strongly attains its norm at the pair $(p, q)$, then again by \ref{lemma:ell1-1_a}, we have $S(g_{i}, p, q) = \sign(a_i)$ and $S(g_{j}, p, q) = -\sign(a_j)$, or we have $S(g_{i}, p, q) = -\sign(a_i)$ and $S(g_{j}, p, q) = \sign(a_j)$. This is a contradiction. 
\end{proof}


\begin{proposition}\label{prop:uncountable_metric}
Let $M$ be a pointed metric space. If $\ell_1$ is isometrically contained in $\sna(M)$, then each basis function must attain its norm at least at uncountably many different pairs of points. In particular, the cardinality of $M$ is at least uncountable.
\end{proposition}

\begin{proof} 
Let $\{g_n\}_n \subseteq \sna(M)$ be isometric to the canonical basis of $\ell_1$ and fix $(a_n)_n \in \ell_1$ with $a_n>0$ for all $n\in \mathbb{N}$. 
Let $\mathcal{S}$ be the family of all sequences $(s_n)_n$ such that $s_n\in\{-1, 1\}$ for all $n\in\mathbb{N}$. 
For each $s=(s_n)_n \in \mathcal{S}$, consider
$f_s=\sum_{n=1}^{\infty} s_n a_n g_n \in \sna (M)$ and say $f_s$ attains strongly its norm at $(r_s, t_s) \in \widetilde{M}$. Without loss of generality, we assume that $S(f_s, r_s, t_s)>0$. By \ref{lemma:ell1-1_a} of Lemma \ref{lemma:ell1-1}, we then see that $S(g_n, r_s, t_s)=\sign(s_n)$ for every $n \in \mathbb{N}$. If $\{s, -s\}\neq \{\ell, -\ell\} \subset \mathcal{S}$, then \ref{lemma:ell1-1_b} of Lemma \ref{lemma:ell1-1} shows that $(r_s, t_s) \neq (r_\ell, t_\ell)$. 
\end{proof} 

Therefore, if $\ell_1$ is isometrically contained in $\sna(M)$, then $M$ must be non-separable (see \ref{S-Nec-Ell1}) and, moreover, the basis functions of $\ell_1$ must attain their norm at an uncountable amount of different places. 

As promised, we observe that none of these conditions are necessary for the set of pointwise norm-attaining Lipschitz functions.

\begin{theorem}\label{Theorem-ell1-pna-1}
There exists a countable pointed metric space $M$ such that $\pna(M)$ contains $\ell_1$ isometrically and, in addition, there exists a sequence of functions $\{f_n\}_{n=1}^{\infty}$ in $\pna(M)$ isometrically equivalent to the canonical basis of $\ell_1$ such that for each $a=(a_n)_n\in\ell_1\backslash \{0\}$, the function $g_a:=\sum_{n=1}^{\infty} a_n f_n$ attains its pointwise norm only at the point $0\in M$.
\end{theorem}

\begin{proof}
Let $S= \{ (r_n)_n \in \{-1, 0, 1\}^{\bbn}:\, \text{there exists } N \in \mathbb{N} \text{ such that } r_k = 0 \text{ for all } k \geq N \}$. Note that $S$ is countable and contains all the possible finite sequences of signs. Let us denote by $\{ s^{(j)} : j \in \mathbb{N} \} = S$. 
Define $M$ as a metric union of spaces as follows:
\[
M = \{0\} \cup \{p_{j, k} : j, k \in \mathbb{N} \} 
\]
endowed with the following metric:
$$d(p,q):=\begin{cases}
0,\quad &\text{if $p=q$}\\
1,\quad &\text{if }p\neq q \text{ and }0\in\{p,q\}\\
1,\quad &\text{if }p=p_{j, k},\, q=p_{j, m} \text{ for some }j\in\bbn,\, m,k\in \bbn,\, m\neq k\\
2,\quad &\text{if }p=p_{j, k},\, q=p_{i , m} \text{ for some }j, i \in\bbn,\, j\neq i,\, m,k\in \bbn.
\end{cases}$$

Fix $c>1$. For each $n\in\bbn$, define the Lipschitz function $f_n:M\rightarrow \bbr$ by
$$f_n(p)=\begin{cases}
s^{(j)}_n \left( 1-c^{-k} \right),\quad &\text{if }p=p_{j, k} \, \text{ for some }k\in\bbn,\, j\in \bbn \\
0,\quad &\text{if }p=0. 
\end{cases}$$

Let $n\in\bbn$, and $(u,v)\in \widetilde{M}$ be fixed. We will compute $|S(f_n, u, v)|$ by considering all the possible cases:
\begin{enumerate}
\item If $u=0$ and $v=p_{j, k}$ for some $j\in\bbn$ and $k\in\bbn$, then
\begin{equation}\label{eq:10k}
|S(f_n, u, v)|= |s^{(j)}_n| \left( 1-c^{-k} \right)\leq \left( 1-c^{-k} \right).
\end{equation} 
Since there exists $s^{(j)}$ for some $j\in \bbn$ such that $s^{(j)}_n\neq 0$, by letting $k \rightarrow \infty$ in \eqref{eq:10k}, we have 
$$\sup_{v\in M\backslash\{0\}} |S(f_n, 0, v)|=1.$$
\item If there exist $j\in\bbn$ and $m> k\in\bbn$ such that $u=p_{j, k}$ and $v=p_{j, m}$, then
$$|S(f_n, u, v)|= |s^{(j)}_n| \left( c^{-k}-c^{-m} \right)\leq \left( c^{-k}-c^{-m} \right)<c^{-k}.$$
In particular, for each $k\in \bbn$
$$\sup_{m\in\bbn\backslash\{k\}} |S(f_n, p_{j, k}, p_{j, m})|\leq c^{-1}.$$
\item If there exist $j,j'\in\bbn$ with $j\neq j'$ and $m \geq k \in\bbn$ such that $u=p_{j, k}$ and $v=p_{j', m}$, then
$$|S(f_n, u, v)|\leq \frac{|s^{(j)}_n|\left( 1-c^{-k} \right)+|s^{(j')}_n|\left( 1-c^{-m} \right)}{2} \leq \frac{2-c^{-k}}{2}.$$
In particular, for each $k\in\bbn$ we get
$$\sup_{j\in\bbn\backslash\{j'\};\, m\in\bbn}|S(f_n, p_{j, k}, p_{j', m})|\leq \frac{2-c^{-1}}{2}.$$
\end{enumerate}
Thus, $f_n$ attains its pointwise norm only at the point $0$ and $\|f_n\| =1$ for each $n \in \mathbb{N}$. 

Let $a=(a_n)_n\in\ell_1\backslash \{0\}$, and let $g_a:=\sum_{n=1}^{\infty} a_n f_n$. We will show that $g_a$ attains its pointwise norm at the point $0$ (and only at that point). First observe that $0$ is the only point at which $g_a$ may attain its pointwise norm: let $x_a\in M$ be such that $g_a$ attains its pointwise norm at $x_a$. Putting $I_a = \{ n \in \mathbb{N} : a_n \neq 0 \}$, 
\begin{align*}
\sum_{n\in I_a} |a_n| = \sup_{y\neq x_a} \frac{|g_a(y) - g_a(x_a)|}{d(x_a, y)} \leq \sum_{n\in I_a} \sup_{y\neq x_a} \left| a_n \frac{f_n(y) - f_n(x_a)}{d(x_a, y)}\right| \leq  \sum_{n\in I_a} |a_n|,
\end{align*}
so all these inequalities must be equalities, which implies that $f_n$ also must attain its pointwise norm at $x_a$ for each $n \in I_a$. 
This shows that $x_a$ should be $0$. 

It remains to check that $g_a$ attains its pointwise norm at $0$. 
Let $\varepsilon>0$ and choose $N \in\bbn$ such that $\sum_{n=1}^{N} |a_n|\geq \|a\|-\varepsilon$. 
Note that there exists some $j\in\bbn$ such that $s^{(j)} \in S$ is defined as follows: 
$$\begin{cases}
s^{(j)}_n=\sign(a_n),\quad &\text{if }1\leq n\leq N \quad \text{and}\quad a_n\neq 0\\
s^{(j)}_n=1,\quad &\text{if }1\leq n\leq N \quad \text{and}\quad a_n= 0\\
s^{(j)}_n=0,\quad &\text{for all }n>N.
\end{cases}$$
It follows that for each $k\in\bbn$,
$$|S(g_s, 0, x_{j, k})|=\sum_{n=1}^{N} |a_n|  \left( 1-c^{-k}\right) \geq (\|a\| - \eps)(1-c^{-k}).$$
This implies that $g_a$ attains its pointwise norm at $0$. 
\end{proof}



\section{Appendix: Other types of norm-attaining Lipschitz functions}\label{Section-NA-Lips}

In this appendix section, we study the relations between pointwise norm-attaining Lipschitz functions and several other existing Lipschitz norm-attainment concepts. We also provide some results on an isometric embedding of $\ell_\infty$ into the set of pointwise norm-attaining Lipschitz functions and the set of Lipschitz functions that attain their norm locally directionally. The results in this appendix will hopefully give a better understanding of the geometrical intuitions underlying the notion of pointwise norm-attainment. 

We start with recalling several norm-attainment concepts from \cite{CCM20, Godefroy16, KMS16}. 

\begin{definition}\label{def:various_definition}
    Let $M$ be a pointed metric space, and let $X$ and $Y$ be real Banach spaces.
\begin{itemize}
\item We say that $f\in\lip(M, Y)$ \textit{attains its norm toward $z\in \|f\| S_Y$} if there exists $\{(p_n, q_n)\}_{n=1}^{\infty}\subset \widetilde{M}$ such that $S(f,p_n,q_n) \longrightarrow z$.
The set of Lipschitz functions $f\in\lip(M, Y)$ that attain their norm toward some vector in $Y$ is denoted by $\lipa(M, Y)$. \label{def:lipa}
\item We say that $f\in\lip(X, Y)$ \textit{attains its norm directionally in the direction $u\in S_X$ toward $z\in \|f\| S_Y$} if there exists $\{(p_n, q_n)\}_{n=1}^{\infty}\subset \widetilde{X}$ such that
$$
S(f,p_n,q_n)\longrightarrow z\quad \text{and}\quad \frac{p_n-q_n}{\|p_n - q_n \|}\longrightarrow u.
$$
The set of Lipschitz functions $f\in\lip(X, Y)$ that attain their norm directionally toward some vector in $Y$ is denoted by $\dira(X, Y)$.
\item We say that $f\in\lip(X, Y)$ \textit{attains its norm locally directionally at a point $\overline{x}$ in the direction $u\in S_X$ toward $z\in \|f\| S_Y$} if there exists $\{(p_n, q_n)\}_{n=1}^{\infty}\subset \widetilde{X}$ such that
$$
S(f,p_n,q_n)\longrightarrow z,\quad \frac{p_n-q_n}{\|p_n-q_n\|}\longrightarrow u\quad \text{and}\quad p_n,q_n\longrightarrow \overline{x}.
$$
The set of Lipschitz functions $f\in\lip(X, Y)$ that attain their norm locally directionally toward some vector in $Y$ is denoted by $\ldira(X, Y)$.
\item We say that $f\in\lip(X, Y)$ \textit{attains its norm at a point $x\in X$ through a derivative in the direction $e\in S_X$} if 
$$
f'(x,e):=\lim_{t\to 0}\frac{f(x+te)-f(x)}{t}\in Y\text{ exists and } \quad\|f'(x,e)\|=\|f\|.
$$
The set of Lipschitz functions $f\in\lip(X, Y)$ that attain their norm at a point through a derivative in a direction $e\in S_X$ is denoted by $\der(X, Y)$.
\label{def:der}
\end{itemize}
\end{definition}

The following relations are known to hold for any infinite metric space $M$ and any real Banach spaces $X$ and $Y$ (see \cite{Choi23,CCM20}): 
\begin{enumerate}[label=(R\arabic*)]
\itemsep0.3em
\item $\sna(M,Y)\subset \lipa(M, Y)$ and $\sna(X,Y)\subset\dira(X,Y)$;
\item $\der(X, Y)\subset\ldira(X, Y)\subset\dira(X, Y)\subset\lipa(X, Y)$;
\item if $Y$ has the Radon-Nikod\'ym property, then $\sna(X, Y)\subset \der(X, Y)$;
\item arguing by compactness, if $Y$ is finite-dimensional, then $\lipa(M,Y)=\lip(M,Y)$ and, if moreover $X$ is finite-dimensional, then $\dira(X, Y)=\lip(X, Y)$. \label{R_compact}
\end{enumerate}


By definition, it is clear that if $f \in \D(X,Y)$ with a point $x \in X$ and in the direction $e \in S_X$, then $f$ attains its pointwise norm at the point $x$. That is, $\D(X,Y) \subseteq \pna (X,Y)$. However, it is not difficult to observe that this inclusion is strict in general. Moreover, we have the following.

\begin{example}\label{ex:der+}
There exists a function $f\in\Lip (\mathbb{R})$ which attains its pointwise norm at $0$ but such that its left and right derivatives at $0$ do not exist. 
\end{example}

\begin{proof}
We shall inductively construct sequences $\{s_n\}_n$ and $\{t_n\}_n$ of points in $\mathbb{R}^2$ as follows:
\begin{itemize}
\itemsep0.3em
\item $t_1 := (t_1 (1), t_1 (2)) := (1, 1/2)\,$ and $s_1 := \left( {t_1(1)} (1/2 + 1/2^3) , t_1(2) \right)$; 
\item $f_1$ be the line with slope $1-1/2^2$ passing through $s_1$.
\end{itemize} 
Suppose that we have constructed the points $s_k$ and $t_k$, and the line $f_{k}$ for $k=1,\ldots, n$ with $n \in \mathbb{N}$. Then 
\begin{itemize}
\itemsep0.3em
\item $t_{n+1} = (t_{n+1} (1), t_{n+1} (2))$ be the intersection point of the lines $f_n$ and $y=x/2$; 
\item $s_{n+1} := \left( {t_{n+1}(1)} (1/2 + 1/2^{n+1} ) , t_{n+1}(2) \right)$; 
\item $f_{n+1}$ be the line with slope $1-{1}/{2^{n+2}}$ passing through $s_{n+1}$. 
\end{itemize} 

Now, let $f:\mathbb{R}\rightarrow\mathbb{R}$ be the odd function with $f(0)=0$ such that the graph of $f |_{\mathbb{R}^+}$ is the union of line segments $[s_n, t_n]$ and $[s_n, t_{n+1}]$ with $n \in \mathbb{N}$ in $\mathbb{R}^2$ and $f(x)=1/2$ for $x\in[1, +\infty[$. Then $\|f\| =1$, $S(f,0,t_n) \rightarrow 1/2$, and $S(f, 0, s_n) \rightarrow 1$. It is clear that $f$ attains its pointwise norm at $0$, but the left and right derivatives of $f$ do not exist at $0$. 
\end{proof}

\vspace{0.3em}

\begin{center}
\begin{figure}[H]
\centering
\includegraphics[scale=6]{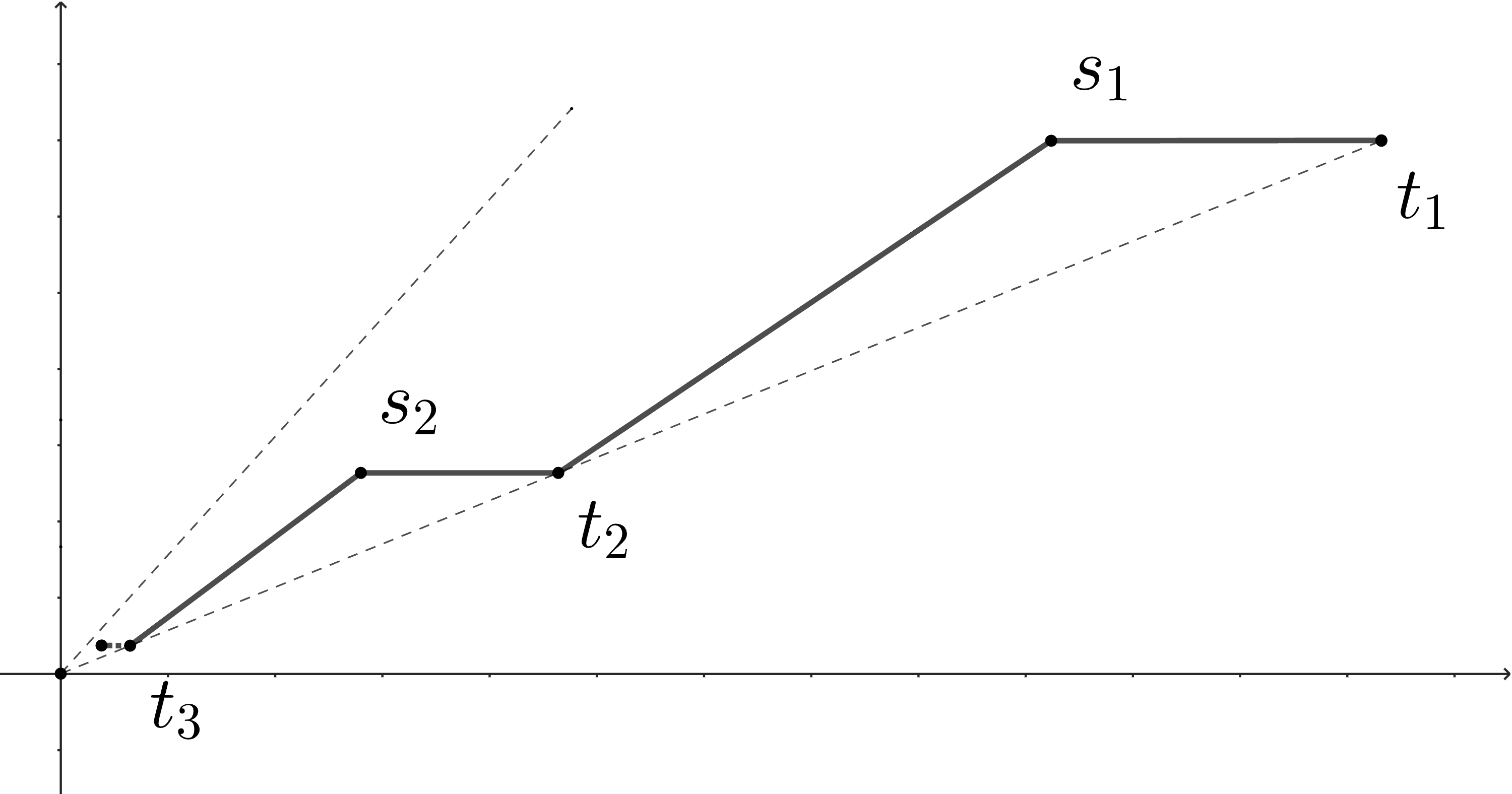}
\caption{The graph over $\bbr^+$ of the function $f \in \Lip (\mathbb{R})$ in Example \ref{ex:der+}}     \label{figure2}
\end{figure}
\end{center}
\vspace{-1em}

Next, we move to a relation between the sets $\pna(\bbr)$ and $\ldira(\bbr)$. Note that both norm-attainment notions occur towards a particular point $p\in \bbr$, and both notions are strictly weaker than $\der(\bbr)$, so it is natural to wonder if there is some relation between them. However, the following examples illustrate that in general, neither set belongs to the other.

\begin{example}\label{ex:ldira-not-pna}
$\ldira\rr$ is not contained in $\pna\rr$. Indeed, consider the function $g \in \Lip ([0,1])$ given in Proposition \ref{thm:ell-infty-in-pna-01} (see Figure \ref{figure}). 
Extend the function $g$ to $\mathbb{R}$ by setting $g(x)=0$ for every $x \not\in [0,1]$. Arguing as in Proposition \ref{thm:ell-infty-in-pna-01}, we deduce that the function $g$ cannot belong to $\pna \rr$. However, the function $g$ belongs to $\ldira \rr$, attaining its norm only towards the point $0$. 
\end{example}

\begin{example}
$\pna(\bbr)$ is not contained in $\ldira(\bbr)$. Indeed, consider the Lipschitz function $f$ on $\mathbb{R}$ defined by $f(x) = \frac{x^2}{|x|+2}  \,\, \text{ for every } x \in \mathbb{R}$. The following claims are all straightforward to check.
\begin{itemize}
\itemsep0.3em
\item $f\in\lip\rr$ with $\|f\|=1$, but $f\notin\ldira\rr$. To show this, observe that
$$f'(x)=\begin{cases}
1-\frac{4}{(x+2)^2},  \quad &\text{if $x>0$}\\
0,\quad &\text{if $x=0$}\\
-1+\frac{4}{(x-2)^2},\quad &\text{if $x<0$}.
\end{cases}$$
This shows our claim as the absolute values of slopes of $f$ are strictly bounded by $1$ around every point.
\item $f$ attains its pointwise norm at every point $p\in\bbr$. Indeed, if $p\geq 0$, then
$$\lim_{n\to\infty}S(f,p,n)=\lim_{n\to\infty}\frac{{n^2}{(n+2)^{-1}}-{p^2}{(p+2)^{-1}} }{n-p}=\lim_{n\to\infty}\frac{n}{n+2}=1.$$
Similarly, if $p<0$, we get
$\lim_{n\to\infty} |S(f, -n, p)| =1.$
\end{itemize}
\end{example}

We finish this section by presenting some $\ell_\infty$-embedding results as promised. 
Recall again from \cite[Theorem 1]{CJ17} that $\ell_\infty$ is isometrically contained in $\lip (M)$ for any infinite metric space $M$. Thus, \ref{R_compact} implies that $\dira(X)$ and $\lipa(M)$ contain an isometric copy of  $\ell_\infty$ for any real Banach space $X$ and any infinite metric space $M$. 

In the following, we shall observe that the sets $\pna (X)$ and $\ldira (X)$ contain an isometric copy of $\ell_\infty$ for any Banach space $X$. Before giving the proofs, we first present the following preliminary lemma. 


\begin{lemma}\label{lemma-extension}
Let $f\in\lip(\bbr)$ be an even function and $X$ a Banach space. If $f \in \pna(\bbr)$ or attains its norm in any sense in Definition \ref{def:various_definition}, 
then the function $\widetilde{f}\in\lip(X)$ given by $\widetilde{f}(x):=f(\|x\|)$ also attains its norm in the same sense as $f$ with $\|\widetilde{f}\|=\|f\|$. 
\end{lemma}

Let us mention that the correspondence $f \in \lip(\bbr) \mapsto \widetilde{f} \in \lip (X)$ in Lemma \ref{lemma-extension} shows, in particular, that whenever the set $\pna(\bbr)$ (respectively, $\ldira(\bbr)$) contains an isometric copy of linear space $Z$ consisting of even functions, then the set $\pna(X)$ (respectively, $\ldira(X)$) also contains $Z$ isometrically. 




\begin{proposition}
For any Banach space $X$, the set $\pna(X)$ contains an isometric copy of $\ell_\infty$ isometrically, while $\pna(X)\neq \lip(X)$.
\end{proposition}

\begin{proof}
Consider the functions $\{f_n\}_{n=1}^\infty \subseteq \pna ([0,1])$ constructed in Proposition \ref{thm:ell-infty-in-pna-01}.
For each $n \in \mathbb{N}$, let $h_n$ be an extension of $f_n$ to $\mathbb{R}$ defined as follows: $h_n(x)=f_n(|x|)$ for every $x \in [-1,1]$, and $h_n(x)=0$ for every $x \not\in [-1,1]$. Then $\{h_n\}_{n=1}^\infty$ is contained in $\pna (\mathbb{R})$ and is isometrically equivalent to the canonical vectors in $\ell_\infty$. Finally, Lemma \ref{lemma-extension} shows that $\pna(X)$ contains an isometric copy of $\ell_\infty$. For the second assertion, take the function $g$ considered in Example \ref{ex:ldira-not-pna}. Define a symmetrization $\widetilde{g}$ of $g$ as $\widetilde{g}(x)= g(|x|)$ and extend it to $X$. It is clear that this extension does not belong to $\pna (X)$.  
\end{proof}






\begin{proposition}
For any Banach space $X$, the set $\ldira(X)$ contains an isometric copy of $\ell_\infty$ isometrically. 
\end{proposition}

\begin{proof}
Let $Y\subset \lip([0,1])$ be isometrically isomorphic to $\ell_\infty$ (as in \cite[Theorem 5]{CJ17}). For each $a\in\ell_\infty$, let $f_a\in Y$ be its corresponding function, so $\|f_a\|=\|a\|$. Define $g_a\in \lip(\bbr)$ as
$$g_a(x):=\begin{cases}
f_a(|x|),\quad &\text{if } |x| \leq 1 \\
f_a(1),\quad &\text{if } |x| \geq 1
\end{cases}$$
Note that 
\begin{equation}\label{eq:g_a}
\sup\{|S(g_a, x, y)|:\, (x,y)\in\widetilde{\bbr}\}=\sup\{|S(g_a, x, y)|:\, (x,y)\in\widetilde{[0,1]}\}
\end{equation} 
since if $x<0$, $0<y<1$, and $z>1$, the following inequalities clearly hold
$$\frac{|g_a(y)-g_a (x)|}{d(x,y)}\leq \frac{|g_a (y)-g_a (0)|}{d(0,y)} \quad \text{and}\quad \frac{|g_a (z)-g_a (y)|}{d(y,z)}\leq \frac{|g_a (1)-g_a (y)|}{d(y,1)}.$$

Now, observe that the mapping $\psi: Y\rightarrow  \lip(\bbr)$ given by $\psi(a):= g_a$ for all $a\in\ell_\infty$ is a linear isometric isomorphism, and every mapping in $\psi(Y)$ is even. We claim that $\psi (Y) \subseteq \ldira(\mathbb{R})$. Let $f\in\psi(Y)$ be given. By \eqref{eq:g_a} we can take two sequences of points $\{p_n\}_n$ and $\{q_n\}_n$ in [0,1] that converge to some points $p$ and $q$ respectively, such that 
$$S(f,p_n,q_n)\rightarrow z\in \{-\|f\|, \|f\|\}.$$
If $p\neq q$, then $f$ strongly attains its norm at the pair $(p, q)$; hence $f \in \ldira(\mathbb{R})$. 
And if $p=q$, $f$ attains its norm locally directionally at the point $p$ by definition. Therefore $\ell_\infty$ is isometrically contained in $\ldira(\bbr)$ consisting only of even functions, and so, it is isometrically contained in $\ldira(X)$ for every Banach space $X$ by Lemma \ref{lemma-extension}.
\end{proof}





\textbf{Acknowledgement}. The authors are grateful to Ram\'on J. Aliaga, Miguel Mart\'{\i}n, and Andr\'es Quilis for fruitful conversations on the topic of the paper. The second author was supported by a KIAS Individual Grant (MG086601) at Korea Institute for Advanced Study. 
The third and fourth authors are supported by Basic Science Research Program, National Research Foundation of Korea (NRF), Ministry of Education, Science and Technology [NRF-2020R1A2C1A01010377].

\end{document}